\documentclass{amsart}

\usepackage{amsmath,amsfonts,amssymb,amscd,verbatim
}

\def\Q{{\mathbb Q}}
\def\Z{{\mathbb Z}}
\def\C{{\mathbb C}}
\def\R{{\mathbb R}}

\def\ib{{\mathbf i}}

                    \def\G{{\mathfrak G}}

\def\SB{{\mathbf S}}

            \def\mult{\mathrm{mult}}
            \def\multB{\mathbf{mult}}

\def\Hol{\mathrm{Aut}}
      
      \def\UU{\mathrm{U}}

\def\Aut{\mathrm{Aut}}
\def\Hom{\mathrm{Hom}}

                                                     \def\eb{\mathrm{e}}

\def\LL{\mathcal{L}}
\def\LB{\mathbb{L}}

\def\IM{\mathrm{Im}}
\def\UU{\mathrm{U}}

\def\Pic{\mathrm{Pic}}
                                 \def\Bim{\mathrm{Bim}}

\def\dim{\mathrm{dim}}

                                   \def\AA{{\mathcal A}}
                                   \def\BB{{\mathcal B}}

                                          \def\Pb{\mathbb{P}}

                                  \def\HH{{\mathbf H}}
                                  \def\EE{{\mathbf E}}

                           \def\uu{{\mathfrak u}}

\newtheorem{thm}{Theorem}[section]
\newtheorem{lem}[thm]{Lemma}
\newtheorem{cor}[thm]{Corollary}

\theoremstyle{definition}
\newtheorem{defn}[thm]{Definition}
\newtheorem{ex}[thm]{Example}

\newtheorem{sect}[thm]{}

           \newtheorem{rem}[thm]{Remark}

\newtheorem{rems}[thm]{Remarks}

\begin{document}\date{}

\title[Complex tori]{Complex tori, theta groups and their Jordan properties}
%and hyperelliptic jacobians}
\author {Yuri G. Zarhin}

\address{Pennsylvania State University, Department of Mathematics, University Park, PA 16802, USA}

\email{zarhin@math.psu.edu}
\thanks{The  author  was partially supported by Simons Foundation Collaboration grant   \# 585711.
Part of this work was done in May-July 2018 when he visited
 the Max Planck Institut f\"ur Mathematik (Bonn, Germany), whose hospitality and support are gratefully acknowledged.}

\subjclass[2010]{14E07; 14K05, 32J18}

\keywords{complex tori, theta groups, Jordan properties}
\begin{abstract}
 We prove that an analogue of Jordan's theorem on finite subgroups of general linear groups does {\sl not} hold for the group of bimeromorphic
  automorphisms of a product of the complex projective line and a complex torus of {\sl positive} algebraic dimension.
\end{abstract}

\maketitle

\section{Introduction}
\begin{sect}
%{\bf Notation}.
As usual, $\C\Pb^1$ denotes the {\sl complex projective line}. %(the Riemann sphere).
 Recall that a group $G$ is called  {\sl Jordan} (V.L. Popov  \cite{Popov}) if there exists a positive integer $J$ that enjoys the following properties. If $B$ is a finite subgroup of $G$ then there exists an abelian normal subgroup $\mathbb{A}$ of $B$ such that the index $[B:\mathbb{A}]\le J$. If this is the case then such a {\sl smallest} $J$ is called the {\sl Jordan constant} of $G$ and denoted by $J_G$; otherwise, we say that $G$ is {\sl not} Jordan and its Jordan constant is $\infty$. V.L. Popov \cite{Popov2} proved that every complex or real Lie group with finitely many connected components is Jordan. (His result also covers the case when the group of connected components is {\sl bounded} \cite{Popov2}.)

If $Z$ is a connected complex  manifold then we write $\Bim(Z)$ for its group of bimeromorphic automorphisms and $\Hol(Z)$ for its subgroup of all biholomorphic automorphisms of $Z$.
Jordan properties of $\Aut(Z)$ and  $\Bim(Z)$ when $Z$ is a compact complex manifolds have been studied recently by Sh. Meng and D.-Q. Zhang \cite{MZ} and Yu. Prokhorov and C. Shramov \cite{PS1,PS2}. In particular, Prokhorov and Shramov have classified all the surfaces with non-Jordan $\Bim$. (The case of projective surfaces was done earlier by V.L. Popov and the author in \cite{Popov,ZarhinE}. See also \cite{Popov2} where Jordan properties of the groups of biholomorphic automorphisms for certain compact and non-compact complex manifolds have been studied.)

The aim of this paper is to study Jordan properties of $\Hol(Y)$ and $\Bim(Y)$ where $Y$ are certain $\C\Pb^1$-bundles over {\sl complex tori}. Recall \cite{Mumford} that a complex compact manifold $X$ is
 a complex torus if it is (biholomorphic to)  a connected {\sl compact complex Lie group}. (Such a group is always {\sl commutative} \cite{Mumford}.) It is  known \cite[Ch. 2, Sect. 6]{BL} that the {\sl algebraic dimension} $\dim_a(X)$ of $X$ is positive if and only if $X$ admits as a quotient-torus a 
{\sl positive-dimensional} complex abelian variety. If $x \in X$ then we write  $T_x \in \Hol(X)$ for the {\sl translation map} 
\begin{equation}
T_x: X \to X, \ z \mapsto z+x \ \forall z \in X.
\end{equation}
Clearly, all $T_x$'s constitute a commutative subgroup  in $\Hol(X)$, because 
$$T_x\circ T_y=T_{x+y} \ \forall x,y \in X.$$
If $\mathcal{V}$ a holomorphic vector bundle over $X$ then for for each $\lambda\in \C^{*}$ we write
$$\mult(\lambda)=\mult_{\mathcal{V}}(\lambda) \in \Hol(\mathcal{V})$$ for the holomorphic automorphism of the total space of $\mathcal{V}$ that acts
as multiplication by $\lambda$ in every fiber. The map
\begin{equation}
\label{multV}
\mult=\mult_{\mathcal{V}}:\C^{*} \to  \Hol(\mathcal{V}), \ \lambda \mapsto \mult_{\mathcal{V}}(\lambda)
\end{equation}
is an {\sl injective group homomorphism}. We write $\Hol_0(\mathcal{V})$ for the centralizer of $\mult_{\mathcal{V}}(\C^{*})$ in  $\Hol(\mathcal{V})$.
Clearly, $\Hol_0(\mathcal{V})$ is a subgroup of $\Hol(\mathcal{V})$ that contains $\mult_{\mathcal{V}}(\C^{*})$.

We write $\mathbf{1}_X$ for the trivial holomorphic line bundle $X \times \C$ on $X$. If $\mathcal{L}$ is a holomorphic line bundle over $X$ then we write $\mathcal{L}_x$ for its fiber over $x\in X$ and $Y_{\mathcal{L}}$ for the  $\C\Pb^1$-bundle over $X$ that is the projectivization $\Pb( \mathcal{L}\oplus \mathbf{1}_X)$ of the rank $2$ vector bundle $\mathcal{L}\oplus \mathbf{1}_X$ over $X$. 
%Clearly, $Y_{\mathcal{L}}$ is canonically biholomorphic to
% the projectivization $\Pb(  \mathbf{1}_X\oplus \mathcal{L})$ of the rank $2$ vector bundle $ \mathbf{1}_X\oplus\mathcal{L}$ over $X$.
\end{sect}

\begin{ex}
\label{rP1}
If $\mathcal{L}=\mathbf{1}_X$ then 
$Y_{\mathcal{L}}=\Pb(\mathbf{1}_X\oplus \mathbf{1}_X)=X\times \C\Pb^1$.
\end{ex}

\begin{sect}
\label{SK}
Let $\LL$ be a holomorphic line bundle over $X$. We write $K(\LL)$ for the set of all $x \in X$ such that $\LL$ is isomorphic to the
{\sl induced} holomorphic line bundle $T_x^{*}\LL$ on $X$. It is  known \cite[pp. 7-8]{Kempf} that $K(\LL)$ is a subgroup of $X$
that coincides with the kernel of a certain holomorphic Lie group homomorphism from $X$ to the dual torus of $X$. This implies
that $K(\LL)$ is a closed (hence compact) complex commutative Lie subgroup in $X$ and therefore has  finitely many connected components. We write
$K(\LL)^{0}$  for the {\sl identity component} of $K(\LL)$; by definition, $K(\LL)^{0}$  is a complex subtorus in $X$,
$$K(\LL)^{0}\subset K(\LL)\subset X;$$
the compactness of $K(\LL)$ implies that the quotient $K(\LL)/K(\LL)^{0}$ is a finite commutative group.

 Let 
us consider the subgroup 
$\mathbf{S}(\LL) \subset \Hol(\LL)$ of all holomorphic automorphisms $\mathfrak{u}$ of the total space of $\LL$ that enjoy the following properties.

\begin{itemize}
\item[(i)]
There exists $x\in X$ such that $\mathfrak{u}: \LL \to \LL$ is a {\sl lifting} of $T_x:X \to X$, i.e., the following diagram is commutative.
\begin{equation}\label{diagram}
\begin{CD}
\LL@>{\uu}>>\LL \\
@V{}VV@V{}VV\\
X@>{T_x}>>X
\end{CD}\end{equation}

\item[(ii)]
For each $z \in X$ the map between the fibers of  $\LL$ over $z$ and $z+x$ induced by $\mathfrak{u}$ is a linear
isomorphism of one-dimensional $\C$-vector spaces.
\end{itemize}
By definition,
\begin{equation}
\label{multSBAut0}
\mult_{\LL}(\C^{*})\subset \mathbf{S}(\LL)\subset\Hol_0(\LL).
\end{equation}
There is a natural group homomorphism 
$$\rho=\rho_{\LL}: \mathbf{S}(\LL)\to X$$ that sends $\mathfrak{u}$ to $x$
if $\mathfrak{u}$ is a lifting of $T_x$. Clearly, $\ker(\rho_{\LL})=\mult_{\LL}(\C^{*})\cong \C^{*}$. This means that $\mathbf{S}(\LL)$ is included in an exact sequence of groups
\begin{equation}
\label{firstexact}
1 \to \C^{*}\overset{\mult_{\LL}}{\longrightarrow}\mathbf{S}(\LL)\overset{\rho_{\LL}}{\longrightarrow} X.
\end{equation}
%each $\lambda\in \C^{*}$ acts as multiplication by $\lambda$ on every fiber of $\LL$.
\end{sect} 

\begin{rem}
\label{invariance}
Let $\psi: \LL_1 \cong \LL_2$ be an isomorphism of holomorphic line bundles $\LL_1$ and $\LL_2$ over $X$. Then 
$K(\LL_1)=K(\LL_2)$, and
all the isomorphisms between $\LL_1$ and $\LL_2$  are of the form 
$ \mult_{\LL_2}(c)\ \psi=\psi \ \mult_{\LL_1}(c)$ where $c$ runs through $\C^{*}$. This implies that the induced by $\psi$ group isomorphism
\begin{equation}
\label{isoSB}
\psi_{\SB}: \SB(\LL_1) \cong \SB(\LL_2), \ \uu \mapsto \psi \uu\psi^{-1}
\end{equation}
does {\sl not} depend on a choice of $\psi$. In addition, 
\begin{equation}
\label{SBcL1L2}
\psi_{\SB}(\mult_{\LL_1}(c))= \mult_{\LL_2}(c) \ \forall  c\in \C^{*}
\end{equation}
and 
$\psi_{\SB}$ may be extended to the commutative diagram
\begin{equation}\label{D}
\begin{aligned}
&&&\SB(\LL_1)                &\stackrel{\psi_{\SB}}\longrightarrow&     &{}           \SB(\LL_2)&&&\\
&&&\rho_{\LL_1}\searrow&                                                     &\swarrow  & \rho_{\LL_2}&&&\\
&&&                               &                    X        {}              &       &                                  &&&
\end{aligned}.
\end{equation}

\end{rem}

In what follows  we write $\#(C)$ for the number of  elements of a finite set $C$. 
By a short exact sequence of complex (resp. real) Lie groups we mean a short exact sequence of groups, each of which is a
 complex (resp. real) Lie group and all the homomorphisms involved are homomorphisms of corresponding complex (resp. real) Lie groups. We do not assume these groups to be connected or to have finitely many connected components.
 
The following assertion  
was inspired by results of D. Mumford \cite[Sect. 23]{Mumford}, who dealt with   abelian varieties.

\begin{thm}
\label{pi0}
If $\LL$ is any holomorphic line bundle over $X$ then
 the group $\SB(\LL)$ carries the natural  structure of a complex Lie group  that 
enjoys the following properties.
\begin{itemize}
\item[(0)] 
The action map
$\SB(\LL)\times \LL \to \LL, \ (\uu,\mathfrak{l})\mapsto \uu(\mathfrak{l}) \ \forall \uu\in \SB(\LL), \mathfrak{l}\in \LL$
of $\SB(\LL)$ on the total space of $\LL$
is holomorphic.
\item[(i)]   $\rho_{\LL}(\mathbf{S}(\LL))=K(\LL)$  and
the short exact sequence of groups induced by \eqref{firstexact}
\begin{equation}
\label{Sexact}
1 \to \C^{*} \overset{\mult_{\LL}}{\to} \mathbf{S}(\LL) \overset{\rho_{\LL}}{\longrightarrow} K(\LL) \to 0
\end{equation}
is a short exact sequence of  complex Lie groups.
\item[(ii)]
Let us consider  
 the preimage
$\mathbf{S}(\LL)^{0}:=\rho_{\LL}^{-1}(K(\LL)^{0})\subset \mathbf{S}(\LL)$, 
 which is a normal clopen complex Lie subgroup
of finite index $\#\left(K(\LL)/K(\LL)^{0}\right)$ in $\mathbf{S}(\LL)$. 
Then  $\mathbf{S}(\LL)^{0}$ is the identity component and the center of $\mathbf{S}(\LL)$.
In particular, 
$\mathbf{S}(\LL)$  is commutative if and only if $K(\LL)$ is connected.
\item[(iii)] If $\psi: \LL \to \LL^{\prime}$ is an isomorphism of holomorphic vector bundles over $X$ then
$\psi_{\SB}:\SB(\LL) \cong \SB(\LL^{\prime})$ defined in \eqref{isoSB} is an isomorphism of complex Lie groups.
\end{itemize}
\end{thm}

\begin{cor}
\label{commutativeTheta}
If $\LL \in \Pic^0(X)$ then $\mathbf{S}(\LL)$  is commutative.
\end{cor}

\begin{proof}[Proof of Corollary \ref{commutativeTheta}]
It is known \cite[Corollary 1.9 on p. 7]{Kempf}  that if $\LL\in \Pic^0(X)$ then $K(\LL)=X$ and therefore is connected.
Now the desired result follows from Theorem \ref{pi0}(ii).
\end{proof}

The following assertion was actually proven in \cite{ZarhinTG} in the case when $\dim(X)=1$. (See also \cite[Cor. 3.6]{BandmanZarhinAG}.)

\begin{thm}
\label{embedP1}
Let $\LL$ be a holomorphic line bundle over $X$. Then
there is a  group embedding 
$$\Upsilon_{\LL}:\mathbf{S}(\LL)\hookrightarrow \Hol(Y_{\mathcal{L}})$$
of  $\mathbf{S}(\LL)$ into the group $\Hol(Y_{\mathcal{L}})$
of holomorphic automorphisms of $Y_{\mathcal{L}}=\Pb( \mathcal{L}\oplus \mathbf{1}_X)$ 
such that the
action map 
$$\mathbf{S}(\LL) \times Y_{\mathcal{L}}\to Y_{\mathcal{L}}, \  (\uu, \mathfrak{y})\mapsto \Upsilon_{\LL}(\uu)( \mathfrak{y})
\ \forall \uu\in \mathbf{S}(\LL) , \mathfrak{y}\in Y_{\mathcal{L}}$$ is holomorphic. In addition,  the action of every $\uu\in \mathbf{S}(\LL)$
on $Y_{\mathcal{L}}$ is a lifting of $T_x: X \to X$ where $x=\rho_{\LL}(\uu)$. In other words, the following diagram is commutative.
\begin{equation}\label{diagram3}
\begin{CD}
Y_{\mathcal{L}}@>{\Upsilon_{\LL}(\uu)}>>Y_{\mathcal{L}} \\
@V{}VV@V{}VV\\
X@>{T_x}>>X
\end{CD}\end{equation}

\end{thm}

The following assertion  
was actually proven in \cite{ZarhinE} in the case when $X$ is an abelian variety and $\LL$ is ample.

\begin{thm}
\label{pi01}
The Jordan constant of $\mathbf{S}(\LL)$   is $\sqrt{\#\left(K(\LL)/K(\LL)^{0}\right)}$.
\end{thm}

We use Theorem \ref{embedP1} and ideas related to Theorem \ref{pi01} in the proof of the following  main result of this paper.

\begin{thm}
\label{manP1}
Let $X$ be a complex torus of positive algebraic dimension. Then the group $\Bim(X\times \C\Pb^1)$ 
is  not Jordan.
\end{thm}

The special case of Theorem \ref{manP1} when $X$ is a complex abelian variety was done in \cite{ZarhinE}.
We also prove the following generalizations of Theorem \ref{manP1}.

\begin{thm}
\label{manP2}
Let $\psi:X\to A$ be a surjective holomorphic group homomorphism from a complex torus $X$ onto a complex abelian variety $A$ of positive dimension.
 Let $\mathcal{F}$ be a holomorphic line bundle on $X$ that enjoys the following property:

 there exist a holomorphic line bundle $\mathcal{M}$ on $A$ and a holomorphic line bundle $\mathcal{F}_0\in \Pic^{0}(X)$ such that $\mathcal{F}$ is isomorphic to the tensor product $\psi^{*}M\otimes \mathcal{F}_0$.

 Then the group $\Bim(Y_{\mathcal{F}})$ is not Jordan.
\end{thm}

\begin{ex}
Taking $X=A$ and $\psi$ the identity map, we obtain from Theorem \ref{manP2} that if $X$ is a positive-dimensional {\sl complex abelian variety}
 then the group $\Bim(Y_{\mathcal{F}})$ is {\sl not} Jordan for every holomorphic line bundle over $X$. (Actually, this assertion follows from results
 of \cite{ZarhinE}.)
\end{ex}

\begin{thm}
\label{manP3}
Let   $X$ be a complex torus and  $\mathcal{F}$ be a holomorphic line bundle on $X$.
Let $X_0$ be a complex subtorus in $X$  that enjoys the following properties.

\begin{itemize}
\item[(i)]
$X_0$ has positive dimension.
\item[(ii)]
The quotient $A:=X/X_0$ is a complex abelian variety of positive dimension.
\item[(iii)]
The restriction of $\mathcal{F}$ to $X_0$ lies in $\Pic^{0}(X_0)$.
\item[(iv)] $\Hom(X_0,A)=\{0\}$.
\end{itemize}
 
 Then the group $\Bim(Y_{\mathcal{F}})$ is not Jordan.
\end{thm}

\begin{ex}
Suppose that $X$ is a two-dimensional complex torus that contains a one-dimensional subtorus $X_0$.
Then $X_0$ is  a one-dimensional abelian variety (elliptic curve) and the quotient $X_1=X/X_0$ is also a
one-dimensional torus and therefore is also an elliptic curve. Now the condition $\Hom(X_0,X_1)=\{0\}$ means that
$X_0$ and $X_1$ are not isogenous.  It follows from Theorem \ref{manP3} that if $X_0$ and $X_1$ are {\sl not} isogenous
and  $\mathcal{F}$ is a holomorphic line bundle on $X$, whose restriction to $X_0$ lies in $\Pic^0(X_0)$ (i.e., has degree $0$)
then $\Bim(Y_{\mathcal{F}})$ is {\sl not} Jordan.
\end{ex}

The paper is organized as follows. In Section \ref{prel} we discuss certain natural nonlinear transformation groups that act in complex vector spaces.
Sections \ref{AppelHumbert}  deals mostly with linear algebra (Hermitian forms,  lattices, discriminant groups) related to holomorphic bundles on complex tori
via {\sl Appel - Humbert theorem}.  In Section  \ref{nonzeroH}  we discuss in detail {\sl theta groups}, which are pretty well known in the case of abelian varieties \cite{Mumford,Kempf}. Theorems  \ref{pi0} and  \ref{embedP1} are proven in Section \ref{proofpi0}. Jordan properties of theta groups are discussed
 in Section \ref{JordanTheta}; they are used 
 in the proof of Theorem  \ref{pi01}
in Section \ref{zeroproof}.  Theorem \ref{manP1} is proven in  Section \ref{firstproof}.
Section \ref{pencils} deals with {\sl pencils} (one-dimensional families) of Hermitian forms; its results are used in Section  \ref{secondproof} in the proofs of Theorems
\ref{manP3} and \ref{manP2}.  In Section \ref{Pic0} we discuss theta groups that correspond to line bundles from $\Pic^0$ and identify them with the complement
of the total space of the line bundle to the zero section. (In particular, we give another proof of their commutativity)

{\bf Acknowledgements}.
This paper is a result of an attempt to  answer  questions of Constantin Shramov. I am grateful to him for interesting stimulating  questions and  discussions. My special thanks go to Vladimir L. Popov, whose thoughtful comments helped to improve the exposition.

\section{Preliminaries}
\label{prel}
\begin{sect}
\label{groupCommutator}
Throughout the paper we will freely use the following well known {\bf commutator pairing} 
\begin{equation}
\label{epairing}
\mathbf{e}: C \times C \to \mathbf{A}
\end{equation}
that arises from  a short exact sequence of groups ({\sl central extension} of $C$ by $\mathbf{A}$) 
\begin{equation}
\label{eCentral}
1 \to \mathbf{A} \to \mathbf{B} \overset{q}{\to}  C \to 1
\end{equation}
where $\mathbf{A}$ is a  central subgroup of $\mathbf{B}$ and $C$ is a commutative group. Recall that in order
to find $\mathbf{e}(c_1,c_2)\in \mathbf{A}$ for $c_1,c_2\in C$ one has to choose {\sl preimages} 
$b_1, b_2 \in \mathbf{B}$ 
with  respect to {\sl surjection} $q:\mathbf{B} \to C$, i.e., 
$$q(b_1)=c_1, q(b_2)=c_2,$$
 and put
\begin{equation}
\label{eDef}
\mathbf{e}(c_1,c_2):=b_1 b_2 b_1^{-1} b_2^{-1} \in \mathbf{A};
\end{equation}
$\mathbf{e}(c_1,c_2)$ does {\sl not} depend on a choice of $b_1,b_2$.
It is well known that $\mathbf{e}$ is bimultiplicative and alternating. It follows from the very definition of $\mathbf{e}$ that a subgroup $\mathbf{K}\subset \mathbf{B}$
is commutative if and only if its image $q(\mathbf{K})$ is an isotropic subgroup of $C$ with respect to $\mathbf{e}$.
\end{sect}
\begin{sect}
Let $V \cong \C^g$ be a finite-dimensional complex vector space of finite positive dimension $g$.
Let $\LB\cong \C$ be a one-dimensional complex vector space and $V_{\LB}:=V \times \LB$ viewed as a complex manifold. We write $\Hol(V_{\LB})$ for the group of holomorphic automorphisms of $V_{\LB}$.  For each $\lambda\in \C^{*}$ we write $\multB(\lambda)$ for the holomorphic automorphism of 
$V_{\LB}$ defined by
$$\multB(\lambda): (v,c) \mapsto (v,\lambda c) \ \forall v \in V, c \in \LB.$$
The map 
$$\multB: \C^{*} \to \Hol(V_{\LB}), \ \lambda \mapsto \multB(\lambda)$$
is an injective group homomorphism with image $\multB(\C^{*})$.  We write  $\Hol_0(V_{\LB})$ for the centralizer of $\multB(\C^{*})$ in $\Hol(V_{\LB})$. 
Clearly, $\Hol_0(V_{\LB})$ is a subgroup of $\Hol(V_{\LB})$  containing $\multB(\C^{*})$.  In what follows we discuss certain subgroups of 
$\Hol_0(V_{\LB})$ related to line bundles on complex tori of the form $V/L$ where $L$ is a  {\sl lattice of maximal rank} in $V$, i.e., a discrete subgroup of rank $2\dim_{\C}(V)=2g$.
\end{sect}
\begin{sect}
\label{BH}
Let $H: V \times V \to \C$ be a Hermitian form on $V$ and
$$E: V \times V \to \R,  \ u,v \mapsto \IM(H(u,v))$$
its imaginary part. Then $E$
is an alternating $\R$-bilinear form such that
\begin{equation}
\label{ib}
E(\ib u,\ib v)=E(u,v), \  H(u,v)= E(u,\ib v)+\ib E(u,v)\ \forall u,v \in V\}.
\end{equation}
As usual, consider the {\sl kernel} of $H$
\begin{equation}
\label{kerHCR}
\ker(H):=\{u\in V\mid H(u,v)=0 \ \forall v\in V\},
\end{equation}
which is a $\C$-vector subspace in $V$.

For each $u \in V$ let us consider $\BB_{H,u} \in \Hol_0(V_{\LB})$ defined as follows.
$$\BB_{H,u}((v,c))=(v+u, \eb^{\pi H(v,u)}c) \ \forall v \in V, c \in \LB.$$
Clearly, $\BB_{H,0}$ is the {\sl identity automorphism} of $V_{\LB}$.
If $u_1, u_2 \in V$ then
$$\BB_{H,u_2}\circ \BB_{H,u_1}((v,c))=\left(v+u_1+u_2, \eb^{\pi H(v+u_1,u_2)} \eb^{\pi H(v,u_1)}c\right)=$$
\begin{equation}
\label{multGcirc}
\left(v+u_1+u_2,  \eb^{\pi H(u_1,u_2)} \eb^{\pi H(v,u_1+u_2)} c\right)=
\multB\left(\eb^{\pi H(u_1,u_2)}\right)\circ \BB_{H,u_1+u_2}((v,c)).
\end{equation}
This implies that in  $\Hol_0(V_{\LB})$ we have
$$\BB_{H,u_1} \BB_{H,u_2}=\multB\left(\eb^{\pi H(u_1,u_2)}\right)\circ \BB_{H,u_1+u_2}$$
and therefore 
$$\BB_{H,u_2}\BB_{H,u_1} \BB_{H,u_2}^{-1}\circ \BB_{H,u_1}^{-1}=\multB\left(\eb^{\pi H(u_1,u_2)}\right)
\left(\multB(\eb^{\pi H(u_2,u_1)})\right)^{-1}=\multB\left(\eb^{2\pi \ib E(u_1,u_2)}\right).$$
In particular, {\sl $\BB_{H,u_1}$ and  $\BB_{H,u_2}$ commute if and only if $E(u_1,u_2)\in \Z$.} 
In addition,  it follows from \eqref{multGcirc} applied to $u_1=u, u_2=-u$ that
\begin{equation}
\label{inverseBB}
\BB_{H,u}^{-1}=\multB(\eb^{-\pi H(u,u)})\circ \BB_{H,-u}, \ 
\left(\multB(\lambda)\circ\BB_{H,u}\right)^{-1}=\multB\left(\frac{\eb^{-\pi H(u,u)}}{\lambda}\right)\circ \BB_{H,-u}.
\end{equation}
We write $\tilde{\G}(H,V)$ for the subset
$$\{\multB(\lambda)\circ\BB_{H,u}\mid \lambda\in \C^{*}, u \in V\}\subset \Hol_0(V_{\LB}).$$
It follows from \eqref{multGcirc} that $\tilde{\G}(H,V)$  is the
subgroup of $\Hol_0(V_{\LB})$ generated by $\multB(\C^{*})$ and all $\BB_{H,u}$ ($u\in V$). Clearly  $\tilde{\G}(H)$ is included in the short exact sequence
\begin{equation}
\label{GHV}
1 \to \C^{*} \overset{\multB}{\longrightarrow} \tilde{\G}(H,V) \overset{\kappa}{\to} V \to 0
\end{equation}
where $\multB(\C^{*})$ is a central subgroup of  $\tilde{\G}(H,V)$ and
the surjective group homomorphism $\kappa: \tilde{\G}(H,V) \to V$ kills 
 $\multB(\C^{*})$ while
  $$\kappa(\BB_{H,u})=u \ \forall u\in V.$$
   In other words, each  $u \in V$ lifts to
$\BB_{H,u}\in \tilde{\G}(H,V)$. This implies that the commutator pairing
$$V \times V \to \multB(\C^{*})$$
attached to central extension \eqref{GHV}
 coincides with
$$u_1, u_2 \mapsto \multB\left(\eb^{2\pi \ib E(u_1,u_2)}\right) \ \forall u_1,u_2 \in V.$$
In particular, if $H=0$ then $E=0$ and $\tilde{\G}(H,V)$ is a {\sl commutative} group.
More generally, if $\Pi \subset V$ is an {\sl additive subgroup} in $V$ then we may define
$\tilde{\G}(H,\Pi)$ as the subgroup of $\Hol_0(V_{\LB})$ generated by $\multB(\C^{*})$ and all $\BB_{H,u}$ ($u\in \Pi$). 
Clearly  $\tilde{\G}(H,\Pi)=\kappa^{-1}(\Pi)$ is included in the short exact sequence
\begin{equation}
\label{GHPi}
1 \to\C^{*} \overset{\multB}{\longrightarrow} \tilde{\G}(H,\Pi) \overset{\kappa}{\to} \Pi \to 0
\end{equation}
where $\multB(\C^{*})$ is a central subgroup of  $\tilde{\G}(H,\Pi) $. Each $u \in \Pi$ lifts to
$\BB_{H,u}\in \tilde{\G}(H,\Pi)$ and the {\sl commutator pairing}
$$\Pi \times \Pi  \to \multB(\C^{*})$$
attached to central extension \eqref{GHPi}
coincides with
$$u_1, u_2 \mapsto \multB\left(\eb^{2\pi \ib E(u_1,u_2)}\right) \ \forall u_1,u_2 \in \Pi.$$
In particular, {\sl if the restriction of $H$ to $\Pi$ is identically $0$  then $\tilde{\G}(H,\Pi)$ is a commutative group.}  

It follows from \eqref{multGcirc} that $\tilde{\G}(H,\ker(H))$ is a {\sl central subgroup} in  $\tilde{\G}(H,V)$ 
\end{sect}
\begin{sect}
The aim of this  subsection and Subsection \ref{realcomplexLie} is to endow  $\tilde{\G}(H,V)$ with the natural structure of a connected {\sl real} Lie group and its certain subgroups $\tilde{\G}(H,\Pi)$ (including $\Pi=\ker(H)$) with the natural structure of a complex Lie group.  Let us consider the complex manifold
$V \times^{H} \C^{*}:=V \times \C^{*}$ endowed with the {\sl composition law}
$$\left(V \times^{H} \C^{*}\right)\times \left(V \times^{H} \C^{*}\right) \to V \times^{H} \C^{*},$$ 
\begin{equation}
\label{groupLAW}
 (u,\lambda), (v,\mu) \mapsto 
(u,\lambda)\circ (v,\mu):=
\left(u+v, \lambda \mu \eb^{\pi H(u,v)}\right).
\end{equation}
The bijectivity of the map
$$\Psi_H: V \times^H \C^{*} \to \tilde{\G}(H,V), \ (u,\lambda)\mapsto \multB(\lambda)\circ \BB_u$$
combined with \eqref{multGcirc} and \eqref{inverseBB} proves that the composition law \eqref{groupLAW} makes 
$V \times^{H} \C^{*}$ a group with identity element $(0,1)$ and the operation of taking the inverse defined by 
 the map
\begin{equation}
\label{inverseLAW}
V \times^{H} \C^{*} \to V \times^{H} \C^{*},
 (v,\lambda)\mapsto  (v,\lambda)^{-1}:= (-u, \eb^{-\pi H(u,u)}/\lambda).
\end{equation}
In addition, $\Psi_H$ is a group isomorphism. Formulas  \eqref{multGcirc} and \eqref{inverseBB} tell us that the group 
$V \times^H \C^{*}$ is actually a real Lie group with real structure induced by the natural complex structure on $V \times^{H} \C^{*}$.
 (However,  if $H=0$ then the real Lie group $V \times^H \C^{*}$ is actually a commutative {\sl complex} Lie group.) Clearly, $V \times^H \C^{*}$
 is included in the short exact sequence of real Lie groups
 \begin{equation}
 \label{GHVLie}
 1 \to \C^{*}\overset{\lambda\mapsto (0,\lambda)}{\longrightarrow} V \times^H \C^{*}\overset{(v,\lambda)\mapsto v}{\longrightarrow}V \to 0.
 \end{equation}
 Applying  $\Psi_H$ to \eqref{GHVLie}, we obtain  that \eqref{GHV} is actually a short exact sequence of real Lie groups.

Let $\Pi$ be a {\sl closed} additive subgroup of $V$. The third theorem of Cartan tells us that $\Pi$ is a real Lie subgroup of $V$. Clearly, 
$$\Pi\times^H \C^{*}:=\Pi\times\C^{*}\subset V\times\C^{*}=V\times^H\C^{*}$$ is a closed subgroup of $V\times^H\C^{*}$
and therefore is its real Lie subgroup. Notice that $\Psi_H(\Pi\times^H \C^{*})=\tilde{\G}(H,\Pi)$, which implies that
$\Psi_H: \Pi\times^H \C^{*} \to \tilde{\G}(H,\Pi)$ is a group isomorphism that provides  $\tilde{\G}(H,\Pi)$ with the structure of a real Lie group. 
Clearly, $\Pi \times^H \C^{*}$
 is included in the short exact sequence of real Lie groups
 \begin{equation}
 \label{GHPiLie}
 1 \to \C^{*}\overset{\lambda\mapsto (0,\lambda)}{\longrightarrow} \Pi \times^H \C^{*}\overset{(v,\lambda)\mapsto v}{\longrightarrow}\Pi \to 0.
 \end{equation}
 Applying  $\Psi_H$ to   \eqref{GHPiLie}, we obtain  that \eqref{GHPi} is actually a short exact sequence of real Lie groups.

\begin{rem}
\label{Pi0real}
Let $\Pi^0$ be the identity component of $\Pi$, which is a connected real Lie subgroup of $V$, i.e., is an $\R$-vector subspace of $V$. 
Then $\Pi^{0}\times^H \C^{*}$  is the connected component of $\Pi \times^H \C^{*}$ that contains the identity element   $(0,1)$ of the group law, i.e.,
the {\sl identity component} of   $\Pi \times^H \C^{*}$. It follows that $\tilde{\G}(H,\Pi^0)$ is the {\sl identity component} of $\tilde{\G}(H,\Pi)$.
This implies that $\Pi/\Pi^0$ is canonically isomorphic to the group $\tilde{\G}(H,\Pi)/\tilde{\G}(H,\Pi^{0})$ of connected components of $\tilde{\G}(H,\Pi)$.
\end{rem}
\end{sect}

\begin{ex}
\label{kerHnice}
If $\Pi=\ker(H)$ then the group law on $\ker(H)\times^H\C^{*}$ is 
$$(u,\lambda), (v,\mu) \mapsto 
(u,\lambda)\circ (v,\mu):=
\left(u+v, \lambda \mu \eb^{\pi H(u,v)}\right)=\left(u+v, \lambda \mu \eb^{\pi \cdot 0}\right)=(u+v,\lambda \mu),$$
since $H(u,v)=0$ for all $u,v\in \ker(H)$.
This means that $\ker(H)\times^H\C^{*}$ is
actually the {\sl direct product} $\ker(H)\times\C^{*}$
of complex Lie groups  $\ker(H)$ and $\C^{*}$; in particular, it is {\sl connected commutative}.
\end{ex}

\begin{sect}
\label{realcomplexLie}
{\sl Now and till the rest of this section let us assume that  $\Pi^{0}=\ker(H)$.}
Since $\ker(H)$ is a complex vector subspace in $V$, it is a complex Lie subgroup in $V$ and therefore $\Pi$ is also a closed complex Lie 
subgroup in $V$. This implies that $\Pi\times^H \C^{*}$ is a closed complex submanifold of $V\times^H\C^{*}$. Recall that  $\Pi\times^H \C^{*}$ is a 
 closed real Lie subgroup of the {\sl real} Lie group $V\times^H\C^{*}$. 
We claim that $\Pi\times^H \C^{*}$ is actually a {\sl complex} Lie group. 

\begin{lem}
\label{real0complex}
If  $\Pi^{0}=\ker(H)$ then the real Lie group $\Pi\times^H \C^{*}$ 
is the complex Lie group with respect to the structure of the complex manifold on $\Pi\times^H \C^{*}$ described above. In addition, \eqref{GHPiLie}  is a short exact sequence of complex Lie groups.
\end{lem}

\begin{proof}
We need to check that the maps \eqref{multGcirc} and \eqref{inverseBB}, being restricted to 
$\left(\Pi\times^H \C^{*}\right)\times \left(\Pi\times^H \C^{*}\right)$ and $\Pi\times^H \C^{*}$ respectively are holomorphic
(not just real analytic).  The complex analyticity condition could be checked locally, in the open neighborhoods 
$$(u_0+\ker(H))\times \C^{*} , (v_0+\ker(H))\times \C^{*}\subseteq \Pi\times^H \C^{*}\times\Pi\times^H \C^{*}$$
of points $(u_0,\lambda_0), (v_0,\mu_0)\in \Pi\times^H \C^{*}\times\Pi\times^H \C^{*}$. Then \eqref{multGcirc} gives us the composition law
$$(u_0+u,\lambda)\circ (v_0+v,\mu) = (u_0+v_0+u+v, \eb^{\pi H(u_0+u,v_0+w)}\lambda\mu)=$$
$$(u_0+v_0+u+v, \eb^{\pi(H(u_0,v_0)}\lambda\mu),$$
which is obviously holomorphic in $u,v, \in \ker(H)$, $\lambda,\mu\in \C^{*}$. Similarly, \eqref{inverseBB} gives us
the operation of taking the inverse
$$(u_0+u, \lambda) \mapsto (u_0+u, \lambda)^{-1}=(-u_0-u, \eb^{-\pi H(u_0+u,u_0+u)}/\lambda)=
(-u_0-u, \eb^{-\pi H(u_0,u_0)}/\lambda),$$
which is obviously holomorphic in $u \in \ker(H)$, $\lambda\in \C^{*}$. The second assertion of Lemma is also obvious.
\end{proof}

\begin{rem}
\label{realcomplexLie1} Let $\Pi^{0}=\ker(H)$.
\begin{itemize}
\item[(i)]
 Lemma \ref{real0complex} and the bijectivity of $\Psi_H$  allow us to endow the {\sl real} Lie group
  $\tilde{\G}(H,\Pi)=\Psi_H(\Pi\times^H\C^{*})$ with the compatible natural structure of a {\sl complex} Lie group, whose {\sl identity component} 
  $$\tilde{\G}(H,\Pi^{0})=\tilde{\G}(H,\ker(H))=\Psi_H(\ker(H)\times^H\C^{*})$$
   is a {\sl central subgroup} of $\tilde{\G}(H,\Pi)$ (and even of $\tilde{\G}(H,V)$). 
   \item[(ii)]
 The {\sl action map}
$$\tilde{\G}(H,\Pi)\times V_{\LB} \to V_{\LB},  \ \Psi_H(u,\lambda), (v,c)\mapsto \mult(\lambda)\circ \BB_u((v,c))=
(v+u, \lambda \eb^{\pi H(v,u)}c)$$
is {\sl holomorphic}. Indeed, it suffices to check that it is holomorphic at all $\Psi_H(u,c)$ from the open subgroup $\tilde{\G}(H,\ker(H))$. In this case $H(v,u)=0$ and we get the map
$\left(\ker(H)\times \C^{*}\right)\times  V_{\LB} \to V_{\LB}, \ (u,\lambda), (v,c)\mapsto (v+u,\lambda c)$, 
which is obviously holomorphic. Clearly,  
\eqref{GHPi} is a short exact sequence of {\sl complex} Lie groups.
\end{itemize}
\end{rem}
\end{sect}

\section{Hermitian forms, lattices, line bundles}
\label{AppelHumbert}
In what follows, a {\sl lattice} is an additive discrete subgroup in a finite-dimensional complex (or real) vector space.

Let $X$ be a positive-dimensional complex torus, i.e., $X=V/L$ where $V\cong \C^g$ a finite-dimensional
complex vector space of positive dimension $g$ and $L\subset V$ a  lattice of maximal possible rank $2g$.
We view $X$ as a connected complex commutative compact Lie group. 
\begin{comment}
 If $y\in X$ then we write
$$T_y: X \to X, \ x \mapsto x+y$$
for the corresponding translation map.
\end{comment}
 By Appel-Humbert theorem \cite{Mumford,Kempf}, holomorphic line bundles $\LL$ on $X$ are 
classified (up to an isomorphism) by {\sl A.-H. data} $(H,\alpha)$ where $H: V \times V \to \C$ is an Hermitian form on $V$ and $\alpha: L \to \UU(1)$ is a map
from $L$ to the {\sl unit circle} $\UU(1)$ that enjoy the following properties:
\begin{equation}
\label{integrality}
E(l_1,l_2):=\IM (H(l_1,l_2))\in \Z, \ 
\alpha(l_1+l_2)=(-1)^{E(l_1,l_2)} \alpha(l_1) \alpha(l_2) \ \forall
l_1, l_2 \in L. 
\end{equation}
We denote by $\LL(H,\alpha)$ the corresponding line bundle on $X$, whose definition is recalled in  Subsection \ref{AHline}.

\begin{sect}
\label{actions}
Let us consider the following  discrete action of the group $L$ on 
 $V_{\LB}$ by holomorphic automorphisms. An element $l \in L$ acts as
$$\AA_{H,l}: V_{\LB} \to V_{\LB}, \ (v,c) \mapsto \left(v+l, \alpha(l)\eb^{\pi H(v,l)+\frac{1}{2}\pi H(l,l)}\cdot c\right) \  \forall v\in V, c\in \C.$$
In other words,
$$\AA_{H,l}=\multB\left(\alpha(l)\eb^{\frac{1}{2}\pi H(l,l)}\right)\BB_{H,l}\in \Hol_0(V_{\LB}).$$
In particular,
$$\AA_{H,l}\in  \tilde{\G}(H,L) \ \forall l \in L.$$
It is well known (and could be easily checked by direct computations) that 
$$\AA_{H,l_1}\circ \AA_{H,l_2}=\AA_{H, l_1+l_2} \ \forall l_1,l_2 \in L.$$
In particular, the map
\begin{equation}
\label{AAL}
\AA^{L}: L \to \Hol_0(V_{\LB}),  \ l \mapsto \AA_{H,l}
\end{equation}
in an {\sl injective} group homomorphism, whose image we denote by 
\begin{equation}
\label{tildeL}
\tilde{L}=\tilde{L}(H,\alpha):=\AA^{L}(L)\subset \tilde{\G}(H,V)\subset \Hol_0(V_{\LB}).
\end{equation}
 Clearly,
$\tilde{L}$ is a subgroup of $\Hol_0(V_{\LB})$ that meets $\multB(\C^{*})$ only at the identity element. In addition,
\begin{equation}
\label{GHL}
\tilde{L}=\tilde{L}(H,\alpha)\subset \tilde{\G}(H,L)\subset \Hol_0(V_{\LB}).
\end{equation}
It is also clear that for each additive subgroup $\Pi\subset V$ we have
\begin{equation}
\label{GHLPi}
\tilde{L}\bigcap  \tilde{\G}(H,\Pi)= \AA^{L}(\Pi\bigcap L)=\{\AA_{H,l}\mid l\in \Pi\bigcap L\}.
\end{equation}
\end{sect}

\begin{sect}
\label{AHline}

The holomorphic line bundle $\LL(H,\alpha) \to X$ is defined \cite[Sect. 2]{Mumford} as the quotient 
$$V_{\LB}/\tilde{L}=V_{\LB}/L \to V/L=X.$$
 In particular, $V_{\LB}/\tilde{L}$ is the total space of the holomorphic line bundle $\LL(H,\alpha)$.
 
 In the obvious notation,
 \begin{equation}
 \label{tensorL}
 \LL(H,\alpha)\otimes \LL(H^{\prime},\alpha^{\prime})\cong 
 \LL(H+H^{\prime},\alpha \alpha^{\prime}).
 \end{equation}
 In particular, $\mathbf{1}_X$ is isomorphic to $\LL(0,\alpha_0)$ where
 \begin{equation}
 \label{alpha0}
 \alpha_0: L \to \{1\}\subset \UU(1)
 \end{equation}
 is the trivial character of $L$.
 \begin{defn}
 One says \cite{Mumford,Kempf,BL} that a holomorphic line bundle on $X$ lies in $\Pic^0(X)$ if it is isomorphic
 to $\LL(0,\alpha)$ for some $\alpha$, i.e., the corresponding Hermitian form is $0$. 
 \end{defn}
\end{sect}

\begin{sect}
\label{thetaL}
We keep the notation and assumptions of Section \ref{AppelHumbert}.
 Since $L$ spans the $\R$-vector space $V$, we have
$$\ker(H)=\{\tilde{x}\in V\mid H(\tilde{x},l)=0 \ \forall l\in L\}.$$
Let us put 
\begin{equation}
\label{LEper}
L_E^{\perp}:= \{\tilde{x}\in V\mid E(\tilde{x},l)\in \Z \ \forall l\in L\}.
\end{equation}
Clearly, $L_E^{\perp}$ is a closed (not necessarily connected)  real Lie subgroup of $V$ that contains $L$ as a discrete subgroup. In particular, the {\sl identity component} $\left(L_E^{\perp}\right)^{0}$ of 
$L_E^{\perp}$ is an $\R$-vector subspace of $V$.
\end{sect}
\begin{lem}
\label{identity}
\begin{itemize}
\item[(i)]
$\left(L_E^{\perp}\right)^{0}=\ker(H)$.
In particular, $\left(L_E^{\perp}\right)^{0}$ is a $\C$-vector subspace of $V$
and $L_E^{\perp}$ is a closed complex Lie subgroup of $V$. 
\item[(ii)]
$\tilde{\G}(H,L_E^{\perp})$ is a complex Lie group that is included in the short exact sequence of complex Lie groups
\begin{equation}
\label{GHLEperp}
1 \to\C^{*} \overset{\multB}{\to} \tilde{\G}(H,L_E^{\perp}) \overset{\kappa}{\to} L_E^{\perp} \to 0
\end{equation}
defined in \eqref{GHPi} for $\Pi=L_E^{\perp}$.

\item[(iii)]
$\tilde{\G}(H,\ker(H))=\kappa^{-1}(\ker(H))$ is the identity component of $\tilde{\G}(H,L_E^{\perp})$,
which is a central clopen subgroup of $\tilde{\G}(H,L_E^{\perp})$ containing $\multB(\C^{*})$ and included
 in the short exact sequence of complex Lie groups
\begin{equation}
\label{GHLieKerH}
1 \to\C^{*} \overset{\multB}{\to} \tilde{\G}(H,\ker(H)) \overset{\kappa}{\to} \ker(H) \to 0 
\end{equation}
defined in \eqref{GHPi} for $\Pi=\ker(H)$
that is induced from \eqref{GHLEperp} by 
 $\ker(H)\subset L_E^{\perp}$.
\item[(iv)]
The action map
$\tilde{\G}(H,L_E^{\perp}) \times V_{\LB} \to V_{\LB}$
is holomorphic.
\end{itemize}
\end{lem}

\begin{proof}
Clearly,
\begin{equation}
\label{V0}
\left(L_E^{\perp}\right)^{0}\subset  \{v\in V\mid E(l,v)=0 \ \forall l\in L\}=:V_0.
\end{equation}
Since $E$ is $\R$-bilinear, $V_0$ is a real vector subspace of $V$.  
In light of first formula of \eqref{ib}, $V_0=\ib V_0$, i.e., $V_0$ is a {\sl complex} vector subspace of $V$.
Since $L$ spans $V$ over $\R$ and $E$ is $\R$-bilinear,
\begin{equation}
\label{V0E}
V_0=\{v\in V\mid E(u,v)=0 \ \forall u \in V\}.
\end{equation}
Since $V_0$ is a $\C$-vector subspace, \eqref{V0E} and second formula of \eqref{ib} imply that
\begin{equation}
\label{V0H}
V_0=\{v\in V\mid H(u,v)=0 \ \forall u \in V\}=\ker(H).
\end{equation}
Now \eqref{V0E} and \eqref{V0H} imply that
$$V_0=\ker(H)\subset \left(L_E^{\perp}\right)^{0},$$
because $\ker(H)$ is connected. In light of \eqref{V0},
$$\ker(H)=V_0\supset \left(L_E^{\perp}\right)^{0}.$$
This implies that $\ker(H)=\left(L_E^{\perp}\right)^{0}$, which 
 proves (i). Now assertions (ii), (iii) and (iv) follow from Remark \eqref{realcomplexLie1}.
\end{proof}

\begin{sect}
\label{normalForm}
Let us put
$$L_0:=L\bigcap \ker(H)=\{l\in L\mid E(l,v)=0 \ \forall v \in V\}=
\{l\in L\mid E(l,m)=0 \ \forall m \in L\}.$$
Then $L_0$ is a free {\sl saturated} $\Z$-submodule  of $L$ and $E$ induces a {\sl nondegenerate} alternating bilinear form on $L/L_0$. In particular, the rank of the free $\Z$-module $L/L_0$ is even. Since the rank of $L$ is even, the rank of $L_0$ is also even. Let $2 g_0$ be the rank of $L_0$. Then the rank of $L/L_0$ is $2(g-g_0)$. Notice also that since $L_0$ is a  lattice in $\ker(H)$,
$$2g_0 \le \dim_{\R}(\ker(H)).$$
Since $L_0$ is saturated in $L$, there exists a saturated free $\Z$-submodule $L_1\subset L$
of rank $2g-2g_0$ such that
$$L=L_0\oplus L_1.$$
This implies that
$$V=L\otimes\R=(L_0\otimes\R)\oplus (L_1\otimes\R).$$
Clearly, the restriction
$$E\bigm|_{L_1}: L_1 \times L_1 \to \Z$$
of $E$ to $L_1$ is a {\sl nondegenerate} alternating bilinear form. This implies that the restriction 
 $$E\bigm|_{L_1\otimes\R}: (L_1\otimes\R)\times  (L_1\otimes\R) \to \R$$
  is a {\sl nondegenerate} alternating $\R$-bilinear form. It follows that
 $$ (L_1\otimes\R)\bigcap \ker(H)=\{0\}$$
 and therefore
 $$2g=\dim_{\R}(V) \ge \dim_{\R}( (L_1\otimes\R))+\dim_{\R}(\ker(H))=$$
 $$ (2g-2g_0)+\dim_{\R}(\ker(H)) \ge   (2g-2g_0)+2g_0 =2g.$$
 This implies that $\dim_{R}(\ker(H))=2g_0$, i.e., $L_0$ is a  lattice of maximal rank in the real vector space $V$ and
 $$\ker(H)=L_0\otimes \R.$$
 \end{sect}
 
 \begin{rem}
 \label{alphaL0}
 It follows from \eqref{integrality} that the restriction of $\alpha$ to $L_0$ is a group homomorphism, i.e.,
 \begin{equation}
 \label{AlphaL0}
 \alpha(l_1+l_2)=\alpha(l_1)\alpha(l_2) \ \forall l_1,l_2\in L_0.
 \end{equation}
 \end{rem}
 
 \section{Theta groups}
 \label{nonzeroH}
 We keep the notation and assumptions of Section \ref{AppelHumbert}.
 \begin{sect}
 \label{normalForm1}
 Suppose that $L\ne L_0$, i.e., $E \ne 0$, i.e., $H\ne 0$. This means that $L_1$ is a free $\Z$-module of {\sl positive} rank $2(g-g_0)$.  Let us  choose once and for all a {\sl basis}
 $\{\bar{l}_1, \dots , \bar{l}_{2g-2g_0}\}$ of $L_1$ and consider the {\sl skew symmetric nondegenerate}  square matrix $\tilde{E}$ of $E\bigm|_{L_1}$ attached to this basis, whose order is  $2g-2g_0$  and entries are
 \begin{equation}
 \label{tildeE}
 \tilde{E}_{ij}:=E\bigm|_{L_1}(\bar{l}_i,\bar{l}_j)=E(\bar{l}_i,\bar{l}_j)\in \Z
 \end{equation}
 The  determinant $\det(E\bigm|_{L_1})$ of $\tilde{E}$  is a nonzero integer that does {\sl not depend} on a choice of the basis. Since $\tilde{E}$  is skew symmetric, $\det(E\bigm|_{L_1})$ is the square 
 of the {\sl pfaffian} of  $\tilde{E}$. Since all the entries of $\tilde{E}$ are integers, its pfaffian is also an integer and therefore $\det(E\bigm|_{L_1})$ is a square in $\Z$; in particular, it is a {\sl positive integer}. Its square root $\sqrt{\det(E\bigm|_{L_1})}$ will play a prominent role in Section \ref{JordanTheta}. On the other hand, $\det(E\bigm|_{L_1})$ 
  admits the following well known  interpretation.
 Let us put
 $$L_{1,E}^{\perp}=\{\tilde{x}\in  L_1\otimes\R\mid E(\tilde{x},l)\in \Z \ \forall l\in L\}.$$
 The nondegeneracy of $E\bigm|_{L_1}$ implies that
 $L_{1,E}^{\perp}$ is a free $\Z$-module of rank $2g-2g_0$ that is contained in $L_1\otimes \Q$ and contains $L_1$ as a subgroup of finite index.  It is well known that 
 \begin{equation}
 \label{indexL1L1}
 [L_{1,E}^{\perp}:L_1]=\#(L_{1,E}^{\perp}/L_1)=\det(E\bigm|_{L_1}).
 \end{equation}
Let us also point out the following obvious but useful equality:
\begin{equation}
\label{LEperp}
 L_{E}^{\perp}=\ker(H)\oplus L_{1,E}^{\perp}=(L_0\otimes\R)\oplus L_{1,E}^{\perp}.
 \end{equation}
In order to get an explicit description of the finite  {\sl discriminant group} $L_{1,E}^{\perp}/L_1$, notice that
  the structure theorem for alternating nongenerate  bilinear forms over $\Z$ implies the existence of a {\sl basis}  
$$e_1,f_1, \dots , e_{g-g_0},f_{g-g_0} \in L_1$$
of the free $\Z$-module $L_1$ 
and positive integers $d_1(E), \dots, d_{g-g_0}(E)$ that enjoy the following properties.
Each $d_i(E)$ divides $d_{i+1}(E)$ (if $1\le i<g_0$),
$$L_1=\oplus_{i=1}^{g-g_0}\left(\Z\cdot e_i\oplus\Z\cdot f_i\right);$$
$$E(e_i,f_j)=0 \ \text{ if } i\ne j; E(e_i,f_i)=d_i(E) \ \forall i.$$
(See  also \cite[pp. 7-8]{BL}.)
It follows that
     $$\det(\tilde{E}) =\left(\prod_{i=1}^{g-g_0}d_i(E)\right)^2,$$
$$L_{1,E}^{\perp}=\oplus _{i=1}^{g-g_0}\frac{1}{d_i(E)}\left(\Z\cdot e_i\oplus\Z\cdot f_i\right).$$
If we define free rank two $\Z$-submodules
$$U_i:= \Z\cdot e_i\oplus\Z\cdot f_i \subset L_1$$
then we get a direct {\sl orthogonal} (with respect to $E$)  splittings
\begin{equation}
\label{splitUs}
L=L_0\oplus L_1, \ L_1= \oplus_{i=1}^{g-g_0}U_i, 
\ L_{1,E}^{\perp}= \oplus_{i=1}^{g-g_0}\frac{1}{d_i(E)} U_i.
\end{equation}

In addition,
\begin{equation}
\label{splitUi}
E\left(\frac{1}{d_i(E)}e_i, \frac{1}{d_i(E)}f_i\right)=\frac{1}{d_i(E)} \ \forall i;  E\left(\frac{1}{d_i(E)}U_i,\frac{1}{d_j(E)}U_j\right)=\{0\} \ \forall i\ne j.
\end{equation}

It follows from \eqref{splitUs}   that
\begin{equation}
\label{splitU}
L_{1,E}^{\perp}/L_1=\oplus _{i=1}^{g-g_0}\left(\frac{1}{d_i(E)}U_i\right)/U_i \cong 
\oplus _{i=1}^{g-g_0}\left(\Z/d_i(E)\Z\right)^2, \ \#(L_{1,E}^{\perp}/L_1)=\left(\prod_{i=1}^{g-g_0}d_i(E)\right)^2=\det(\tilde{E}).
\end{equation}
It also follows from \eqref{splitUs} that
\begin{equation}
\label{XL12}
X=V/L\supset  L_E^{\perp}/L=(\ker(H)/L_0)\oplus \left(L_{1,E}^{\perp}/L_1\right)=
(\ker(H)/L_0)\oplus\left( \oplus _{i=1}^{g-g_0}\left(\frac{1}{d_i(E)}U_i\right)/U_i\right).
\end{equation}
(See  also \cite[pp. 7-8]{BL}.)
\end{sect}
\begin{rem}
\label{pi0K}
Suppose that $H \ne 0$.
It follows from \cite[pp. 7-8]{Kempf} that
\begin{equation}
\label{KLalpha}
K(\LL(H,\alpha))=L_E^{\perp}/L\subset X.
\end{equation}
  Now \eqref{XL12} implies that $\ker(H)/L_0$ is the {\sl identity component}  $K(\LL(H,\alpha))^{0}$ of the complex Lie group $K(\LL(H,\alpha))$ while
the group $K(\LL(H,\alpha))/K(\LL(H,\alpha))^{0}$ 
 is isomorphic to $L_{1,E}^{\perp}/L_1$
and therefore 
\begin{equation}
\label{KLalpha3}
\#\left(K(\LL(H,\alpha))/K(\LL(H,\alpha))^{0}\right)=\#\left(L_{1,E}^{\perp}/L_1\right)=\det(E\bigm|_{L_1})=\left(\prod_{i=1}^{g-g_0}d_i(E)\right)^2.
\end{equation}
\end{rem}

\begin{sect}
\label{finiteGroupPairing}
Let us consider the alternating bilinear pairing
\begin{equation}
\label{pairingL1}
\mathbf{e}_E: L_{1,E}^{\perp}/L_1\times L_{1,E}^{\perp}/L_1 \to \C^{*},
\  (v_1+L_1, v_2+L_1)\mapsto \eb^{2\pi\mathbf{i}E(v_1,v_2)} \ \forall v_1+L_1,v_2+L_1 \in L_{1,E}^{\perp}/L_1.
\end{equation}
It follows from \eqref{splitUi} and \eqref{splitUs} that $\mathbf{e}_E$ is a {\sl nondegenerate pairing}.
\end{sect}

\begin{lem}
\label{divn}
Suppose that $H \ne 0$. Let $n$ be a positive integer such that
\begin{equation}
\label{nE}
E(l_1,l_2)\in n\Z \ \forall l_1,l_2 \in L.
\end{equation}
Then  $\#(L_{1,E}^{\perp}/L_1)$ is divisible by $n^2$.
\end{lem}

\begin{proof}
Since $H \ne 0$, we have $g_0 <g$, i.e. $g-g_0 \ge 1$. It follows from \eqref{tildeE} that all the entries of the order $2(g-g_0)$ square matrix $\tilde{E}$  
are divisible by $n$ in $\Z$ and therefore $\det(\tilde{E})$ is divisible by $n^{2(g-g_0)}$ in $\Z$. This implies that 
$\#(L_{1,E}^{\perp}/L_1)=\det(\tilde{E})$  is divisible by $n^{2(g-g_0)}$
 and therefore is divisible by $n^2$.
\end{proof}

\begin{thm}
\label{quotient}
\begin{itemize}
\item[(i)]
$\tilde{L}=\tilde{L}(H,\alpha)$ is a central discrete subgroup of $\tilde{\G}(H,L_E^{\perp})$ that  meets $\multB (\C^{*})$ only at the identity.
\item[(ii)]
$$\tilde{L}_0:=\tilde{L}\bigcap \tilde{\G}(H,\ker(H))=\AA^{L}(L_0)=\{\AA_{H,l}\mid l\in L_0\}$$
is a discrete subgroup in the commutative connected complex Lie group $\tilde{\G}(H,\ker(H))$.
\item[(iii)]
 The commutative connected complex Lie group  $\tilde{\G}(H,\ker(H))/\tilde{L}_0$ is isomorphic to the
 quotient $\left(\ker(H)\times \C^{*}\right)/\bar{L}_0$ where 
 $\bar{L}_0:=\{(l,\alpha(l))\mid l\in L_0\}$ is a discrete subgroup in $\ker(H)\times \C^{*}$.
\end{itemize}
\end{thm}

\begin{proof}
We have already seen that $\tilde{L}$ meets $\multB (\C^{*})$ only at the identity
and $\tilde{L}\subset \tilde{\G}(H,L)$. Since $L_E^{\perp}$ contains $L$, we have
$$ \tilde{\G}(H,L)=\kappa^{-1}(L)\subset \kappa^{-1}\left(L_E^{\perp}\right)= \tilde{\G}\left(H,L^{\perp}\right)$$
and therefore  $\tilde{L}\subset \tilde{\G}\left(H,L^{\perp}\right).$
Recall that $E(L_E^{\perp},L)\subset \Z$.  So, if $\tilde{l}\in \tilde{L}, \phi \in \tilde{\G}(H,L_E^{\perp})$ then $E\left(\kappa(\tilde{l}),\kappa(\phi)\right)\in \Z$ and therefore $\tilde{l}$
and  $\phi$ commute (see the very end of Subsect. \ref{BH}). This implies that $\tilde{L}$ is a central subgroup of $\tilde{\G}\left(H,L_E^{\perp}\right)$. 
In order to check the discreteness of $\tilde{L}$, recall that $\kappa(\tilde{L})=L$ is a discrete subgroup in $L_E^{\perp}$. Hence, if $\tilde{L}$ is {\sl not} discrete,
the intersection $\tilde{L}\bigcap \ker(\kappa)$ is infinite. However, $\ker(\kappa)=\multB (\C^{*})$ and we know that  $\tilde{L}$ meets $\multB (\C^{*})$ only at a single element.
The obtained contradiction proves that $\tilde{L}$ is discrete. This proves (i). Since $L_0=L\bigcap \ker(H)$, (ii) follows from (i) combined with \eqref{GHLPi} applied to $\Pi=\ker(H)$.
Now (iii) follows from (ii) combined with Example \ref{kerHnice}.
\end{proof}

\begin{rem}
The same arguments prove that  $\tilde{\G}(H,L)$
is a central subgroup of $\tilde{\G}\left(H,L_E^{\perp}\right)$. In fact, 
the natural ``multiplication map'' 
$$\multB(\C^{*})\times \tilde{L}(H,\alpha)\to \tilde{\G}(H,L)$$
is a group isomorphism.
\end{rem}

\begin{sect}
\label{thetaG}
Applying  short exact sequence \eqref{GHPi} to $\Pi=L_E^{\perp}$, we get a short exact sequence of complex Lie groups
\begin{equation}
\label{GHL1}
1 \to \C^{*} \overset{\multB}{\longrightarrow}  \tilde{\G}\left(H,L_E^{\perp}\right) \overset{\kappa}{\to} L_E^{\perp} \to 0
\end{equation}
where the image $\multB(\C^{*})$ is a central subgroup of  $\tilde{\G}\left(H,L_E^{\perp}\right)$. Each $u \in L_E^{\perp}$ lifts to
$\BB_{H,u}\in \tilde{\G}\left(H,L_E^{\perp}\right)$ and the {\sl commutator pairing}
$$L_E^{\perp} \times L_E^{\perp}  \to \multB(\C^{*})$$
attached to \eqref{GHL1}
 coincides with
$$u_1, u_2 \mapsto \multB\left(\eb^{2\pi \ib E(u_1,u_2)}\right) \ \forall u_1,u_2 \in L_E^{\perp}.$$
Recall that 
$$\tilde{L}\subset\tilde{\G}\left(H,L_E^{\perp}\right) \subset \tilde{\G}(H,V)\subset \Hol_0(V_{\LB}),$$
 is a central discrete subgroup $\tilde{L}$ of $\tilde{\G}\left(H,L_E^{\perp}\right)$ that acts discretely on  $V_{\LB}$
 and 
 $$\LL(H,\alpha)=V_{\LB}/\tilde{L}=V_{\LB}/\tilde{L}(H,\alpha).$$
This gives us the natural embedding of   the complex Lie quotient group
$$\G(H,\alpha):=\tilde{\G}\left(H,L_E^{\perp}\right)/\tilde{L}=\tilde{\G}\left(H,L_E^{\perp}\right)/\tilde{L}(H,\alpha)$$ into the group $\Hol(\LL(H,\alpha))$ of holomorphic automorphisms of 
the total space of $\LL(H,\alpha)$.
Further, we will identify $\G(H,\alpha)=\tilde{\G}\left(H,L_E^{\perp}\right)/\tilde{L}$ with its (isomorphic) image in $\Hol(\LL(H,\alpha))$. 
It follows from Lemma \ref{identity}(iii) that the action map
$$\G(H,\alpha)\times \LL(H,\alpha) \to \LL(H,\alpha)$$
is holomorphic.  
\end{sect}

\begin{sect}
\label{thetaD}
It follows from \eqref{GHL1} and \eqref{KLalpha} that
$\G(H,\alpha)=\tilde{\G}\left(H,L_E^{\perp}\right)/\tilde{L}$ is included in a short exact sequence of complex Lie groups
\begin{equation}
\label{GHL2}
1 \to \C^{*} \to \G(H,\alpha)\overset{\bar{\kappa}}{\to} L_E^{\perp}/L  \ (=K(\LL(H,\alpha)\ )\to 0.
\end{equation}
Here the image of $\C^{*}$ is a central subgroup in $\G(H,\alpha)$, each 
$$\lambda \in \C^{*}\to \G(H,\alpha) \subset \Hol(\LL(H,\alpha))$$
acts on the total space of $\LL(H,\alpha)$ as the multiplication by $\lambda$ at every fiber of $\LL(H,\alpha) \to X$,
i.e, $\lambda$ is mapped to $\mult_{\LL(H,\alpha)}(\lambda)$;
 the surjective complex Lie group homomorphism 
$$\bar{\kappa}:\G(H,\alpha)=\tilde{\G}\left(H,L_E^{\perp}\right)/\tilde{L}\to  L_E^{\perp}/L$$
 kills  the image of $\C^{*}$ and sends a coset
$\BB_{H,u} \tilde{L}$ to $u+L$ for each $u \in L^{\perp}$.

 Clearly, the  {\sl commutator pairing} 
attached to central extension \eqref{GHL2} is
\begin{equation}
\label{thetaCOMM}
\mathbf{e}_{H,\alpha}: L^{\perp}/L  \times L^{\perp}/L  \to \C^{*},  \
u_1+L, u_2+L \mapsto \eb^{2\pi\ib E(u_1,u_2)} \ \forall
u_1+L, u_2+L \in  L^{\perp}/L.
\end{equation}

\begin{rem}
\label{restrict}
\begin{itemize}
\item[(i)]
Clearly, the restriction of $\mathbf{e}_{H,\alpha}$ to $L_{1,E}^{\perp}/L_1
\times L_{1,E}^{\perp}/L_1$ coincides with the {\sl nondegenerate} pairing 
\eqref{pairingL1}.
\item[(ii)]
Clearly, $\ker(H)/L_0$ lies in the kernel of $\mathbf{e}_{H,\alpha}$. Combining this with (i), we obtain that $\ker(H)/L_0$ coincides with the kernel of $\mathbf{e}_{H,\alpha}$, since
 $$L_E^{\perp}/L=(\ker(H)/L_0)\oplus (L_{1,E}^{\perp}/L_1).$$

\end{itemize}
\end{rem}
\end{sect}

\begin{thm}
\label{GHalpha0}
The identity component $\G(H,\alpha)^{0}$ of $\G(H,\alpha)$ coincides 
with the preimage $\kappa^{-1}(\ker(H)/L_0)$ of 
$$\ker(H)/L_0\subset  L_E^{\perp}/L\subset V/L=X$$ and is canonically isomorphic as a complex Lie group
to the quotient $\tilde{\G}(H,\ker(H))/\tilde{L}_0$. In particular,  $\G(H,\alpha)^{0}$  is a central
subgroup of $\G(H,\alpha)$ that is isomorphic as a complex Lie group to $\left(\ker(H)\times\C^{*}\right)/\bar{L}_0$.
\end{thm}

\begin{proof}
It follows from Theorem \ref{quotient} that 
$\G(H,\alpha)^{0}$ is the image of $\tilde{\G}(H,\ker(H))$ in $\tilde{\G}\left(H,L_E^{\perp}\right)/\tilde{L}=\G(H,\alpha)$
and this image coincides with
$$\tilde{\G}(H,\ker(H))/\left(\tilde{L}\bigcap \tilde{\G}(H,\ker(H))\right)=\tilde{\G}(H,\ker(H))/\tilde{L}_0\subset \tilde{\G}\left(H,L_E^{\perp}\right)/\tilde{L}=\G(H,\alpha).$$
Since $\ker(H)/L_0$ is the identity component of $L_E^{\perp}/L$,  \eqref{GHL2} implies that
$\G(H,\alpha)^{0}\subset \kappa^{-1}(\ker(H)/L_0)$.
The connectedness of $\C^{*}$ implies that its image in $\G(H,\alpha)$ (see \eqref{GHL2}) lies in $\G(H,\alpha)^{0}$.
The exactness of \eqref{GHL2} implies that in order to prove the desired equality, it suffices to check that for each $u+L_0\in \ker(H)/L_0$ (with $u\in \ker(H)$), there is
$\uu\in \G(H,\alpha)^{0}$ with $\bar{\kappa}(\uu)=u+L_0$. Thanks to \eqref{GHPi}  and \eqref{GHL1},
$$\uu:=\BB_{H,u}\tilde{L}_0\in \tilde{\G}(H,\ker(H))/\tilde{L}_0=\G(H,\alpha)^{0}$$
does the trick. Now the last assertion of Theorem follows from Theorem \ref{quotient}(iii).
\end{proof}

\begin{thm}
\label{abelianSUB}
Let $\tilde{B}$ be a subgroup of $\G(H,\alpha)$ and 
$$B:=\bar{\kappa}(\tilde{B})\subset  L^{\perp}/L$$
be its image, which is a subgroup in $L^{\perp}/L$.

\begin{itemize}
\item[(0)]
If $N$ is a subgroup of $B$ then 
$$\tilde{N}:=\bar{\kappa}^{-1}(N)\bigcap \tilde{B}=
\{u \in \tilde{B}\subset \G(H,\alpha)\mid \bar{\kappa}(u)\in B\}$$
is a normal subgroup in $\tilde{B}$.
In addition, if  $\tilde{B}$ is finite then 
 $[B:N]$ and $[\tilde{B}:\tilde{N}]$ coincide.

\item[(i)]
$\tilde{B}$ is commutative if and only if $B$ is isotropic with respect to $\mathbf{e}_{H,\alpha}$.
\item[(ii)]
Suppose that $H \ne 0$ and 
$$B\subset L_{1,E}^{\perp}/L_1\subset L^{\perp}/L.$$
Then
$\tilde{B}$ is commutative if and only if $B$ is isotropic with respect to $\mathbf{e}_E$.
If this is the case then  the index
$\left[\left(L_{1,E}^{\perp}/L_1\right):B\right]$ is divisible by 
$\sqrt{\#\left(L_{1,E}^{\perp}/L_1\right)}$. 
 \item[iii)]
Suppose that $H \ne 0$ and $n$ is a positive integer such that
$$E(L,L)\subset n\Z.$$
Let $\tilde{A}$ be a commutative subgroup  of $\G(H,\alpha)$ and let
$$A:=\bar{\kappa}(\tilde{A})\subset  L^{\perp}/L$$
be its image, which is a subgroup in $L^{\perp}/L$.
If $A\subset  L_{1,E}^{\perp}/L_1$ then the index $\left[\left(L_{1,E}^{\perp}/L_1\right):B\right]$ is divisible by $n$.
\item[iv)]
Suppose that $H \ne 0$ and
\begin{equation}
\label{projection}
\mathrm{pr}_2: L^{\perp}/L=(\ker(H)/L_0)\oplus (L_{1,E}^{\perp}/L) \to (L_{1,E}^{\perp}/L_1)
\end{equation}
be the projection map. Let us put
$$B_1:=\mathrm{pr}_2(B)=\mathrm{pr}_2\bar{\kappa}(\tilde{B})\subset L_{1,E}^{\perp}/L_1.$$
Then $\tilde{B}$ is commutative if and only if $B_1$ is isotropic with respect to $\mathbf{e}_E$.
If this is the case then the index $[(L_{1,E}^{\perp}/L_1):B]$ is divisible by 
$\sqrt{\#(L_{1,E}^{\perp}/L_1)}$. 
\end{itemize}
\end{thm}

\begin{proof}
(0). Since $B\subset L_{1,E}^{\perp}/L_1$ is commutative, its every subgroup, including $N$, is normal in $B$.
Let us consider the {\sl surjective} homomorphism 
\begin{equation}
\label{BC}
\bar{\kappa}:\tilde{B} \to B
\end{equation}
 and denote its kernel by $\tilde{B}_0$, which is a normal subgroup in $\tilde{B}$. The surjectivity of \eqref{BC} implies that
the preimage $\tilde{N}\subset \tilde{B}$ of $N$ is also normal in $\tilde{B}$; in addition, $\tilde{N}$ contains $\tilde{B}_0$, which is normal in $\tilde{N}$. The surjection \eqref{BC} induces  group isomorphisms 
$$\tilde{B}/\tilde{B}_0\cong B, \ \tilde{N}/\tilde{B}_0\cong N.$$
If $\tilde{B}$ is finite then all the other groups involved are also finite and
$$\#(\tilde{B})=\#(\tilde{B}_0)\cdot\#(B), \ \#(\tilde{N})=\#(\tilde{B}_0)\cdot\#(N),$$
which implies that
$$[\tilde{B}:\tilde{N}]=\frac{\#(\tilde{B})}{\#(\tilde{N})}=\frac{\#(\tilde{B}_0)\cdot\#(B)}{\#(\tilde{B}_0)\cdot\#(N)}=\frac{\#(B)}{\#(N)}=[B:N].$$
It follows that $[\tilde{B}:\tilde{N}]=[B:N]$. This completes the proof of (0).

(i) follows from the description  \eqref{thetaCOMM} of $\mathbf{e}_{H,\alpha}$ as the commutator pairing attached to the central extension \eqref{GHL2}.

The first assertion of (ii) follows from (i) and Remark \ref{restrict}. The second assertion of (ii) follows from the first one and the {\sl nondegeneracy} of  $\mathbf{e}_E$.

(iii) follows from (ii) combined with Lemma \ref{divn}.

(iv) By (i), $\tilde{B}$ is commutative if and only if $B$ is isotropic with respect $\mathbf{e}_{H,\alpha}$.
Let 
$$x_1, x_2 \in  B\subset L^{\perp}/L=\ker(H)/L_0\oplus L_{1,E}^{\perp}/L_1\subset X.$$
We have
$$x_1=h_1+l_1, x_2=h_2+l_2; \ h_j\in \ker(H)/L_0, \ l_j\in L_{1,E}^{\perp}/L_1.$$
By Remark \ref{restrict}, each $h_j$ lies in the kernel of 
$\mathbf{e}_{H,\alpha}$. This implies that
$$\mathbf{e}_{H,\alpha}(x_1,x_2)=\mathbf{e}_{H,\alpha}(l_1,l_2)=\mathbf{e}_E(l_1,l_2).$$
This implies that $B$ is isotropic with respect to $\mathbf{e}_{H,\alpha}$ if and only if 
$B_1$ is isotropic with respect to $\mathbf{e}_E$. The remaining assertion about the index follows from the {\sl nondegeneracy} of  $\mathbf{e}_E$.
\end{proof}

\begin{sect}
\label{defTheta}
We call   $\G(H,\alpha)$ the {\sl theta group} of $\LL(H,\alpha)$. Recall (Remark \ref{pi0K}) that
\begin{equation}
\label{KLperp}
K(\LL(H,\alpha))=L^{\perp}/L\subset V/L=X.
\end{equation}
Clearly, 
 $$\G(H,\alpha)\subset \mathbf{S}(\LL(H,\alpha)).$$ More precisely, all elements  of $\mult_{\LL(H,\alpha)}(\C^{*})$ are liftings of the identity automorphism 
$T_e$ of $X$ (where $e$ is the zero of group law on $X$)
while $\BB_{H,u} \tilde{L}$  is a lifting of  $T_{x}$ where
$$x=u +L \in  L_E^{\perp}/L =K(\LL(H,\alpha))\subset V/L=X.$$
It follows that 
$$\bar{\kappa}: \G(H,\alpha) \to  L_E^{\perp}/L =K(\LL(H,\alpha))\subset X$$ coincides with the restriction of $$\rho=\rho_{\LL( H,\alpha)}: \mathbf{S}(\LL( H,\alpha)) \to X$$ (defined in Subsection \ref{SK})
to $\G(H,\alpha)\subset \mathbf{S}(\LL(H,\alpha))$.
\end{sect}

\begin{thm}
\label{theta0}
The identity component $\G(H,\alpha)^{0}=\bar{\kappa}^{-1}(\ker(H)/L_0)$ of the complex Lie group $\G(H,\alpha)$ 
 is the center of $\G(H,\alpha)$, which is included in the short exact sequence of complex Lie groups
$$1 \to \C^{*}\to \G(H,\alpha)^{0}\overset{\bar{\kappa}}{\to} \ker(H)/L_0 \to 0.$$
In particular,  $\G(H,\alpha)$ is commutative if $H=0$.
\end{thm}

\

\begin{proof}[Proof of Theorem \ref{theta0}]

It follows from the results of Subsection \ref{thetaD} that $\mathfrak{u} \in \G(H,\alpha)$  lies in the {\sl center} of $\G(H,\alpha)$ if and only if 
$$x:=\bar{\kappa}(\mathfrak{u}) \in L_E^{\perp}/L$$
satisfies
$$x=v+L \ \text{ with } v \in  L_E^{\perp}, \ \eb^{2\pi\ib E(v,w)}=1 \ \forall w \in L_E^{\perp},$$
i.e.,
\begin{equation}
\label{intCondition}
E(v,w)=\IM(H(v,w))\in \Z \ \forall w \in L_E^{\perp}.
\end{equation}
Clearly,  each $v \in \ker(H)$ satisfies \eqref{intCondition} and therefore the center of $\G(H,\alpha)$ contains
$$\bar{\kappa}^{-1}\left(\ker(H)/\left(\ker(H)\bigcap L\right)\right)=\bar{\kappa}^{-1}\left(\ker(H)/L_0\right)=\G(H,\alpha)^{0}.$$
In particular, if $H=0$ then 
$$\ker(H)=V, \  L_0=L, \  L_E^{\perp}=V,   \ \ker(H)/L_0=L_E^{\perp}/L;$$
hence  $\G(H,\alpha)$ coincides 
 with its central subgroup $\G(H,\alpha)^{0}$ and therefore is commutative.

Now suppose that   (in the notation of Subsection \ref{normalForm})
$$v \not\in \ker(H)\oplus\left(\oplus_{i=1}^{g-g_0} U_i\right).$$
This implies that $H \ne 0$ and  there exist $u_0 \in \ker(H)$ and $u_i \in \frac{1}{d_i}U_i$ (for all $i$ with $1 \le i \le g-g_0$) such that {\sl not all} $u_i \in U_i$ and
$$u=u_0+\sum_{i=1}^{g-g_0}u_i.$$
Suppose that $u_j \not\in U_j$ for  a certain $j \in \{1, \dots, g-g_0\}$. Then
$u_j=a_j e_j+b_j f_j$ where $a_j, b_j \in \frac{1}{d_j}\Z$ and, at least, one of $a_j, b_j$ is {\sl not an integer}. Recall that 
$$\frac{1}{d_j}e_j, \frac{1}{d_j}f_j\in L_E^{\perp}.$$
However,
$$E\left(u, \frac{1}{d_j}e_j\right)=b_j, \ E\left(u, \frac{1}{d_j}f_j\right)=a_j$$
and therefore \eqref{intCondition} does not hold. This implies that if $\mathfrak{u}$ is a {\sl central} element
of $\G(H,\alpha)$
then
$v \in \ker(H)\oplus \left(\oplus_{i=1}^{g-g_0} U_i\right)$, i.e., 
$$\bar{\kappa}(\mathfrak{u})\in \ker(H)/L_0,$$
which means that $\mathfrak{u} \in \G(H,\alpha)^{0}$.
This completes the proof.
\end{proof}

\begin{thm}
\label{thetaS}
If  $\LL=\LL(H,\alpha)$ then $\G(H,\alpha)= \mathbf{S}(\LL( H,\alpha))$.
In particular,
% $\rho_{\LL(H,\alpha)}=\bar{\kappa}$ and 
\begin{equation}
\label{rhoKappa}
\rho_{\LL(H,\alpha)}=\bar{\kappa}, K(\LL(H,\alpha))=\bar{\kappa}(\G(H,\alpha))=\rho_{\LL(H,\alpha)}(\mathbf{S}(\LL( H,\alpha))).
\end{equation}
\end{thm}

\begin{rem}
\label{pi3K}
Recall (Remark \ref{pi0K}) that $\ker(H)/L_0$ is the identity component of  $K(\LL(H,\alpha))$.
%This implies that 
%$$\G(H,\alpha)^{0}\subset \bar{\kappa}^{-1}(\ker(H)/L_0).$$
%The connectedness of $\G(H,\alpha)^{0}$ implies that it contains the image of $\C^{*}$.
% $K(\LL(H,\alpha))^{0}$ of $K(\LL(H,\alpha))$. 
Now results of Subsection  \ref{defTheta}
combined with Theorems \ref{thetaS}  imply that
\begin{equation}
\label{centerST}
\G(H,\alpha)^{0}=\bar{\kappa}^{-1}(\ker(H)/L_0)=\rho^{-1}\left(K(\LL(H,\alpha))^{0}\right)=\mathbf{S}(\LL(H,\alpha))^{0}.
\end{equation}
\end{rem}

\begin{proof}[of Theorem \ref{thetaS}]
We write  $p:\LL \to X$ for the structure morphism from the total space of the holomorphic line bundle to its base.

Let $\mathfrak{u} \in  \mathbf{S}(\LL)\subset \Hol(\LL)$. By definition of $\mathbf{S}(\LL)$, there is $x \in X$ such that 
$\mathfrak{u}: \LL \to \LL$ is a lifting of $T_x: X\to X$. In particular, 
the restriction of $\mathfrak{u}$ to  the fibers of  $\LL$ induces  the linear isomorphisms 
$$\mathfrak{u}_z: \LL_z \cong  \LL_{z+x}$$
between the fibers of $\LL$ at $z$ and $x+z$ for all $z\in X$.

 It follows from  \cite[Ch. 1, Sect. 2,  proposition 2.14]{Wells}
(applied to $f=T_x$)
that there exist an {\sl induced} holomorphic line bundle  
$$T_x^{*}\LL=\{(\mathfrak{l},z)\in \LL \times X\mid  p(\mathfrak{l})=z+x\}$$ over  $X$
with the structure morphism
$$T_x^{*}\LL \to X, \ (\mathfrak{l},z) \mapsto z,$$
  and a  holomorphic map  of total spaces of holomorphic line bundles over $X$
$$(T_x)_{*}: T_x^{*}\LL \to \LL, \  (\mathfrak{l},z) \mapsto \mathfrak{l}$$ 
that lifts $T_x$ and induces $\C$-linear isomorphisms between the corresponding fibers
$(T_x^{*}\LL)_z$ and $\LL_{z+x}$ for all $z\in X$. Clearly, $(T_x)_{*}$ is a biholomorphic isomorphism: indeed, 
its inverse is defined by 
$$\mathfrak{l} \mapsto  (\mathfrak{l},z)= (\mathfrak{l}, p(\mathfrak{l})-x).$$
It follows that the composition
$$(T_x)_{*}\circ\mathfrak{u}^{-1}: \LL \to \LL \to T_x^{*}\LL$$
is an isomorphism of holomorphic line bundles $\LL$ and $T_x^{*}\LL$ over $X$.  
Therefore holomorphic line bundles $\LL( H,\alpha)=\LL$ and  $T_x^{*}\LL( H,\alpha)=T_x^{*}\LL$ over $X$ are isomorphic.
 Hence
 $$x \in K(\LL( H,\alpha))= L_E^{\perp}/L \subset V/L=X.$$
 Pick $u \in L_E^{\perp}$ with $u+L=x$. Then $\mathfrak{u} \BB_{H,u}^{-1}$ is a holomorphic automorphism of  $\LL( H,\alpha)$ that leaves invariant every fiber $\LL(H,\alpha)_z$ and acts on it as a $\C$-linear automorphism. By compactness and connectedness of $X$, there is a nonzero scalar $\lambda\in \C^{*}$ such that $\mathfrak{u} \BB_{H,u}^{-1}$ acts as multiplication by $\lambda$ on every fiber. It follows that $\mathfrak{u} \BB_{H,u}^{-1}$ lies in $\G(H,\alpha)$. Since $\BB_{H,u}$ lies in $\G(H,\alpha)$ as well, we conclude that $\mathfrak{u} \in \G(H,\alpha)$.
\end{proof}

\begin{rem}
\label{ShortST}
It follows from Theorem \ref{thetaS} combined with \eqref{GHL2} and the results of Subsection \ref{defTheta} that
$\mathbf{S}(\LL( H,\alpha))=\G(H,\alpha)$ is included in a short exact sequence of complex Lie groups
\begin{equation}
\label{SrhoTheta}
1 \to \C^{*} \to \mathbf{S}(\LL( H,\alpha)) (=\G(H,\alpha)) \overset{\rho=\bar{\kappa}}{\longrightarrow} K(\LL(H,\alpha))\to 0.
\end{equation}
\end{rem}

\section{Proofs of Theorem \ref{pi0} and \ref{embedP1}}
\label{proofpi0}
\begin{defn}
\label{defSBLie}
Let $\LL$ be a holomorphic line bundle on $X$. Let us choose an isomorphism
of holomorphic line bundles $\psi:\LL \cong \LL(H,\alpha)$ for suitable A.H. date $(H,\alpha)$ where $(H,\alpha)$ is uniquely determined by the isomorphism class of $\LL$.
By Remark \ref{invariance} combined with Theorem \ref{thetaS},  there is a certain group isomorphism
$\psi_{\SB}:\SB(\LL)\cong \SB(\LL(H,\alpha)=\G(H,\alpha)$ that does {\sl not} depend on a choice of $\psi$.
Then there is the canonical structure of a complex Lie group on $\SB(\LL)$ such that the group isomorphism $\psi_{\SB}:\SB(\LL)\cong  \G(H,\alpha)$
is an isomorphism of complex Lie groups.
\end{defn}

\begin{cor}
\label{SBholAction}
The action map $\SB(\LL)\times \LL \to \LL$ and the group homomorphism $\rho_{\LL}: \SB(\LL)\to X$ are holomorphic.
\end{cor}

\begin{proof}
We may assume that $\LL=\LL(H,\alpha)$ and therefore $\SB(\LL)=\G(H,\alpha)$. Then our assertion follows from the
results of Subsection \ref{thetaG} and Theorem \ref{thetaS}.
\end{proof}

\begin{proof}[Proof of Theorem \ref{pi0}]
(0)  and (i) are contained in Corollary \ref{SBholAction} and Theorem \ref{thetaS}. (iii) follows from the very Definition \ref{defSBLie}.
In order to check (ii), let us assume that $\LL=\LL(H,\alpha)$ and therefore 
$$\SB(\LL)=\G(H,\alpha),  \rho_{\LL}=\bar{\kappa}, \SB(\LL)^{0}=\G(H,\alpha)^{0}, K(\LL)^{0}=\ker(H)/L_0.$$
Then all the assertions of (ii) follow from Theorems \ref{GHalpha0} and Theorem \ref{theta0}.  
%except the claim that $\G(H,\alpha)^{0}$ is the center of $\G(H,\alpha)$:
%we only know that it is a central subgroup. We prove this claim in 
\end{proof}

\begin{comment}
\begin{thm}
\label{embedP1}
Let $\LL$ be a holomorphic line bundle over $X$. Then
there is a  group embedding 
$$\mathbf{S}(\LL)\hookrightarrow \Hol(Y_{\mathcal{L}})$$
of  $\mathbf{S}(\LL)$ into the group $\Hol(Y_{\mathcal{L}})$
of holomorphic automorphisms of $Y_{\mathcal{L}}=\Pb( \mathcal{L}\oplus \mathbf{1}_X)$.
%such that the
%action map $\mathbf{S}(\LL) \times Y_{\mathcal{L}}$ is holomorphic. In addition, if $\uu\in \mathbf{S}(\LL)$
\end{thm}
\end{comment}

\begin{proof}[Proof of Theorem \ref{embedP1}]
%We may assume that $\LL=\mathcal{L}(H,\alpha)$. By Theorem \ref{thetaS},
%$$\mathbf{S}(\LL)=\mathbf{S}( \LL(H,\alpha))=\G( H,\alpha).$$
Denote by $\mathcal{V}$ the rank 2 vector bundle
$\mathcal{V}=\LL\oplus \mathbf{1}_X$ over $X$. By definition, 
 $Y_{\LL}$ is the projectivization of 
$\mathcal{V}$.

First, let us define an embedding
$$\SB(\LL) \hookrightarrow \Hol_0(\mathcal{V})=\Hol_0( \LL\oplus \mathbf{1}_X ).$$
%to the group $\Hol_0( \LL\oplus \mathbf{1}_X)$ of holomorphic $\C$-linear automorphisms of the total space of $\mathcal{V}$.
%$ \mathcal{L}( H,\alpha)\oplus \mathbf{1}_X $. 
In order to do that, recall that  each $\mathfrak{u} \in \SB(\LL) \subset \Hol_0(\LL)$ is a lifting of  
 of $T_x: X \to X$ where $x=\rho_{\LL}(\mathfrak{u})\in X$.  This allows us to define the action of 
$\mathfrak{u}$ on $\mathbf{1}_X=X \times \mathbb{C}$ as 
$$\bar{\kappa}_1(\mathfrak{u}): X \times \mathbb{C} \to X \times \mathbb{C}, (z,\lambda)\mapsto (z+\rho_{\LL}(\mathfrak{u}),\lambda) \ \forall z\in X, \lambda \in \C.$$
By Corollary \ref{SBholAction}, $\rho_{\LL}$ is a homomorphism of complex Lie groups, hence is holomorphic and therefore the corresponding  action map
$$\SB(\LL)\times \mathbf{1}_X \to \mathbf{1}_X,  \ \uu, (z,\lambda)\mapsto (z+\rho_{\LL}(\mathfrak{u}),\lambda) $$
is holomorphic.
 This gives us a (non-injective) group homomorphism
$$\bar{\kappa}_1: \SB(\LL)  \to \Hol_0(\mathbf{1}_X),$$
whose image meets ``scalar automorphisms''
$\mult_{\mathbf{1}_X}(\C^{*})$ only at the identity automorphism of $\mathbf{1}_X$.
Taking the ``direct sum'' of  the embedding  $\SB(\LL) \subset \Hol_0(\LL)$
with $\bar{\kappa}_1$, we get a group {\sl embedding}
$$\bar{\kappa}_2:\SB(\LL) \hookrightarrow \Hol_0( \LL\oplus \mathbf{1}_X)=\Hol_0(\mathcal{V}),$$
whose image also meets precisely one element of $\mult_{\mathcal{V}}(\C^{*})$, namely the {\sl identity automorphism} of $\mathcal{V}$. 
Clearly, the corresponding action map 
$$\SB(\LL)\times \mathcal{V}\to \mathcal{V},  \ \uu, \left(\mathfrak{l}_z; (z,\lambda)\right)\mapsto \left(\uu(\mathfrak{l}_z); (z+\rho_{\LL}(\mathfrak{u},\lambda)\right)
\ \forall z\in X, \mathfrak{l}_z\in \LL_z, \lambda\in \C, \uu\in \SB(\LL)$$
 is holomorphic as well, since the action map $\SB(\LL)\times \LL \to \LL$ is holomorphic,
thanks to Corollary \ref{SBholAction}. It is also clear that $\bar{\kappa}_2(\uu)$ is a lifting of $T_x$ with $x=\rho_{\LL}(u)$.
It follows that the  group homomorphism
$$\Upsilon_{\LL}:\SB(\LL)  \to \Hol(\Pb(\mathcal{V}))=\Hol(Y_{\LL})$$
induced by $\bar{\kappa}_2$
is an embedding, the corresponding action map $\SB(\LL) \times Y_{\LL} \to Y_{\LL}$ is holomorphic
and $\Upsilon_{\LL}(\uu)$ is a a lifting of $T_x$ with $x=\rho_{\LL}(u)$.
\end{proof}

\section{Jordan properties of theta groups}
\label{JordanTheta}
We keep the notation and assumption of Section \ref{AppelHumbert}.

\begin{thm}
\label{JordanT}
Suppose that $H \ne 0$. Then $\G(H,\alpha)$ is Jordan and its Jordan constant
 is % $\sqrt{\#(L_{1,E}^{\perp}/L_1)}$.
$$\sqrt{\#(L_{1,E}^{\perp}/L_1)}=\prod_{i=1}^{g-g_0}d_i(E).$$
\end{thm}

\begin{cor}
\label{JordanN}
Let $H\ne 0$ and $n$ be a positive integer such that
$$E(L,L)\subset n\mathbb{Z}.$$
Then the Jordan constant $J_{\G(H,\alpha)}$ of $\G(H,\alpha)$ is divisible by $n$. In particular, 
$J_{\G(H,\alpha)}\ge n$.
\end{cor}

\begin{proof}[Proof of Corollary \ref{JordanN}]
By Lemma \ref{divn}, $\#(L_{1,E}^{\perp}/L_1)$ is divisible by $n^2$.
Now the desired result follows readily from Theorem \ref{JordanT}.
\end{proof}

We will need the following Lemma that will be proven at the end of this section.

\begin{lem}
\label{liftFinite}
Let $\Delta$ be a finite subgroup in $K(\LL(H,\alpha))$.  Then there exists a finite subgroup $\tilde{\Delta}$ in $\G(H,\alpha)$ such that 
$\bar{\kappa}(\tilde{\Delta})=\Delta$.
\end{lem}

\begin{proof}[Proof of Theorem \ref{JordanT}]
Let $\tilde{B}$ be a finite subgroup in $\G(H,\alpha)$. Let us consider its images
$$B=\bar{\kappa}(\tilde{B})\subset K(\LL(H,\alpha))=(\ker(H)/L_0)\oplus (L_{1,E}^{\perp}/L_1), \
B_1=\mathrm{pr}_2(B)=\mathrm{pr}_2 \bar{\kappa}(\tilde{B})\subset L_{1,E}^{\perp}/L_1.$$
Let $A_1$ be a {\sl maximal isotropic  subgroup} in $B_1$ with respect to $\mathbf{e}_E$.
The nondegenerate pairing $\mathbf{e}_E$ gives rise to an embedding
$$B_1/A_1 \hookrightarrow \Hom(A_1,\C^{*}), b+A_1 \mapsto \{a \mapsto \mathbf{e}_E(a,b) \ \forall a\in A_1\}.$$
Since the orders of finite commutative groups $A_1$ and $\Hom(A_1,\C^{*})$ coincide, 
$\#(B_1/A_1)$ divides $\#(A_1)$ and therefore $\#(B_1/A_1)^2$ divides $\#(B_1)$, which in turn, divides $\#(L_{1,E}^{\perp}/L_1)$. It follows that the index
$$[B_1:A_1]=\#(B_1/A_1)$$
{\sl does not exceed} (actually divides) $\sqrt{\#(L_{1,E}^{\perp}/L_1)}$.
Let us consider the subgroup
$$\tilde{A}:=(\mathrm{pr}_2 \bar{\kappa})^{-1}(A_1)\bigcap \tilde{B}.$$
Since $A_1$ is isotropic, it follows from Theorem \ref{abelianSUB}(iv)  that $\tilde{A}$ is a commutative subgroup. Since $A_1$ is obviously normal in (commutative) $B_1$,
the preimage $\tilde{A}$ of $A_1$ with respect to {\sl surjective}
$$\tilde{B}\overset{\mathrm{pr}_2\bar{\kappa}}{\longrightarrow} B_2$$
is a normal subgroup of  $\tilde{B}$, whose index does not exceed (actually equals)
$[B_1:A_1]$, which, in turn, does not exceed $\sqrt{\#(L_{1,E}^{\perp}/L_1)}$. It follows that $\G(H,\alpha)$ is Jordan and its Jordan constant does not exceed $\sqrt{\#(L_{1,E}^{\perp}/L_1)}$. We need to prove that the Jordan constant is, at least, $\sqrt{\#(L_{1,E}^{\perp}/L_1)}$.

In order to do that, let us  consider the finite subgroup
$$\Delta:=L_{1,E}^{\perp}/L_1=\{0\}\oplus\left(L_{1,E}^{\perp}/L_1\right)  \subset (\ker(H)/L_0)\oplus (L_{1,E}^{\perp}/L_1)=K(\LL(H,\alpha)).$$
By Lemma \ref{liftFinite}, there is a finite subgroup $\tilde{\Delta}\subset \G(H,\alpha)$ such that
$$\bar{\kappa}(\tilde{\Delta})=\Delta.$$
Let $A^{\prime} \subset \tilde{\Delta}$ be a {\sl commutative} normal subgroup of $\tilde{\Delta}$. By Theorem \ref{abelianSUB}(ii), the subgroup
$$A=\bar{\kappa}(A^{\prime})\subset \Delta= L_{1,E}^{\perp}/L_1$$
is  an isotropic subgroup  with respect to $\mathbf{e}_E$ and the index $[\left(L_{1,E}^{\perp}/L_1\right):A]$ is divisible by $\sqrt{\#(L_{1,E}^{\perp}/L_1)}$.    Let us define
$$\tilde{A}:=\bar{\kappa}^{-1}(A)\bigcap \tilde{\Delta} \subset  \tilde{\Delta}.$$
By Theorem \ref{abelianSUB}{0}, $\tilde{A}$ is a normal subgroup of $\tilde{\Delta}$ and
$$[\tilde{\Delta}:\tilde{A}]=[L_{1,E}^{\perp}/L_1:A].$$
This implies that $[\tilde{\Delta}:\tilde{A}]$ is divisible by $\sqrt{\#(L_{1,E}^{\perp}/L_1)}$.

Clearly, $\tilde{A}$ contains $A^{\prime}$. This implies that $[\tilde{\Delta}:A^{\prime}]$ is divisible by $[\tilde{\Delta}:\tilde{A}]$ and therefore is divisible by $\sqrt{\#(L_{1,E}^{\perp}/L_1)}$. However, if $U$ is a {\sl maximal isotropic subgroup} in  $L_{1,E}^{\perp}/L_1$ then
$$\#(U)=\sqrt{\#(L_{1,E}^{\perp}/L_1)}=[L_{1,E}^{\perp}/L_1:U].$$
Let  us put
$$\tilde{U}:=\bar{\kappa}^{-1}(U)\bigcap \tilde{\Delta} \subset  \tilde{\Delta}.$$
By Theorem \ref{abelianSUB}(0,ii),  $\tilde{U}$ is a commutative normal subgroup in 
$L_{1,E}^{\perp}/L_1$ of index $\sqrt{\#(L_{1,E}^{\perp}/L_1)}$.
It follows that the Jordan constant of $\G(H,\alpha)^{0}$ is, at least, $\sqrt{\#(L_{1,E}^{\perp}/L_1)}$.  This  completes the proof.
%implies  that the Jordan constant of $\G(H,\alpha)^{0}$ is $\sqrt{\#(L_{1,E}^{\perp}/L_1)}$.

\end{proof}

\begin{proof}[ Proof of Lemma \ref{liftFinite}]
In what follows we identify  $\C^{*}$ with its image in $\G(H,\alpha)$ and view it as a certain central  subgroup of $\G(H,\alpha)$.
Let $d$ be the exponent of $\Delta$. Let us consider  the finite multiplicative group $\mathbf{\mu}_{d}$ of all $d$th roots of unity and the finite multiplicative group $\mathbf{\mu}_{d^2}$ of all $d^2$th roots of unity in $\C$. We have
$$\mathbf{\mu}_{d} \subset \mathbf{\mu}_{d^2} \subset \C^{*}\subset \G(H,\alpha).$$

For each $x \in \Delta$ choose its lifting $\tilde{x} \in \G(H,\alpha)$ with the same order as $x$
and such that the lifting $\widetilde{(-x)}$ of $-x$ coincides with $\tilde{x}^{-1}$.
 (This is possible, since  $\C^{*}$ is a central divisible subgroup in $\G(H,\alpha)$.)  Let us consider the finite set
$$\tilde{\Delta}=\{\gamma \tilde{x}\mid \gamma \in \mathbf{\mu}_{d^2}, x \in \Delta\}\subset \G(H,\alpha).$$
Clearly, $\bar{\kappa}(\gamma \tilde{x})=x$ and therefore $\bar{\kappa}(\tilde{\Delta})=\Delta$.  It remains to check that $\tilde{\Delta}$ is a subgroup in  $\G(H,\alpha)$.  Let $x_1,x_2 \in \Delta$ and $x_3=x_1 +x_2\in \Delta$. We need to compare $\tilde{x_1}\tilde{x_2}$ 
and $\tilde{x_3}$ in  $\G(H,\alpha)$. Notice that there is $\gamma \in \C^{*}$ such that
$$\tilde{x_3}=\gamma \tilde{x_1}\tilde{x_2}.$$
In addition,
$$\tilde{x_1}^d=\tilde{x_2}^d=\tilde{x_3}^d=1  \in \C^{*}\subset  \G(H,\alpha).$$
On the other hand, we have 
$$\gamma_0:=\tilde{x_1}\tilde{x_2}\tilde{x_1}^{-1}\tilde{x_2}^{-1}\in \mathbf{\mu}_d  \in \C^{*}\subset  \G(H,\alpha),$$
since the orders of  both $\tilde{x_1}$ and $\tilde{x_2}$  divide $d$. It follows that the images of $\tilde{x_1}$ and $\tilde{x_2}$ in the quotient
$\G(H,\alpha)/\mathbf{\mu}_d$ do commute and therefore the image of 
$\tilde{x_1}\tilde{x_2}$ in $\G(H,\alpha)/\mathbf{\mu}_d$ has order dividing $d$. This means that
$$(\tilde{x_1}\tilde{x_2})^d \in \mathbf{\mu}_d$$
and therefore 
$$(\tilde{x_1}\tilde{x_2})^{d^2}=1.$$
It follows that
$$1=\tilde{x_3}^{d^2}=\left(\gamma \tilde{x_1}\tilde{x_2}\right)^{d^2}=
\gamma^{d^2}(\tilde{x_1}\tilde{x_2})^{d^2}=\gamma^{d^2}\cdot 1=1.$$
This implies that $\gamma^{d^2}=1$ and therefore
$$\tilde{x_1}\tilde{x_2}=\gamma^{-1}\tilde{x_3}\in \tilde{\Delta}.$$
It follows  that $\tilde{\Delta}$ is a subgroup. (See also \cite[Sect. 4, p. 132, ex. 3]{BourbakiAI}.)
\end{proof}

\section{Proof of Theorem   \ref{pi01}}
\label{zeroproof}

We may and will assume that $\LL=L(H,\alpha)$.
We keep the notation and assumptions of  Section \ref{JordanTheta}.

%\begin{proof}[Proof of Theorem \ref{pi01}]
By Theorem \ref{thetaS},
$\mathbf{S}(\LL)=\G(H,\alpha)$. By Theorem \ref{theta0},
$$\G(H,\alpha)^{0}:=\bar{\kappa}^{-1}(\ker(H)/L_0)$$
is the center of $\G(H,\alpha)$.

Suppose that $H=0$. Then 
 $K(\LL(H,\alpha))=X$  \cite[Cor. 1.9 on p. 7]{Kempf}; in particular, it is connected, i.e., the number of its connected components is $1$.
On the other hand, by Theorem \ref{theta0}, $\G(H,\alpha)$ is commutative
and therefore its Jordan constant is $1$. This gives us the desired result when $H = 0$.

Suppose that $H \ne 0$. By Theorem \ref{JordanT}, the Jordan constant of $\G(H,\alpha)$
is $\sqrt{\#\left(L_{1,E}^{\perp}/L_1\right)}$. By Remark \ref{pi0K},
$L_{1,E}^{\perp}/L_1$ is isomorphic to the group \newline 
$K(\LL(H,\alpha))/K(\LL(H,\alpha))^{0}$ of connected component of $K(\LL(H,\alpha))$.
This implies that the Jordan constant of $\G(H,\alpha)$ is $\sqrt{\#\left(K(\LL(H,\alpha))/K(\LL(H,\alpha))^{0}\right)}$.
 This completes the proof.
%\end{proof}

\section{$\C\Pb^1$-bundles over complex tori}
We start with the following elementary but useful observations that allow us to handle the groups of bimeromorphic automorphisms of $\C\Pb^1$-bundles, using an information about the groups of biholomorphic automorphisms.

\begin{rems}
\label{section}
Let $\mathcal{L}$ and $\mathcal{N}$ be  holomorphic line bundles over $X$. Assume that $\mathcal{L}$ admits a {\sl nonzero} holomorphic section say, $s$.  Let $n$ be a positive integer.

\begin{itemize}
\item[(0)]
Clearly,   $\mathcal{L}^n$ also admits a {\sl nonzero} holomorphic section $s^{\otimes n}$.
\item[(i)]
The holomorphic $\C$-linear map of rank 2 holomorphic vector bundles on $X$
$$ \mathcal{N} \oplus  \mathbf{1}_X \to  (\mathcal{N}\otimes\mathcal{L})\oplus\mathbf{1}_X, \
(t_x; x,\lambda) \mapsto (t_x\otimes  s(x);  x,\lambda) \ \forall x \in X, t_x \in \mathcal{N}_{x}, \lambda\in \C$$
induces a bimeromorphic isomorphism of the corresponding 
%(projectivizations)
 $\C\Pb^1$-bundles $\Pb( \mathcal{N}\oplus \mathbf{1}_X)=Y_{\mathcal{N}}$ and  $\Pb((\mathcal{N}\otimes \mathcal{L})\oplus \mathbf{1}_X)
=Y_{\mathcal{N}\otimes \mathcal{L}}$ over $X$. Therefore the groups    $\Bim(Y_{\mathcal{N}})$  and $\Bim(Y_{\mathcal{N}\otimes \mathcal{L}})$  are isomorphic.

Taking into account that  $\mathcal{L}^n$ also admits a {\sl nonzero} holomorphic section, we obtain that  the groups    $\Bim\left(Y_{\mathcal{N}}\right)$  and 
$\Bim(Y_{\mathcal{N}\otimes \mathcal{L}^n})$  are isomorphic  for all positive integers $n$. 

\item[(ii)]
It follows from  (i) applied to $\mathcal{N}=\mathbf{1}_X$ combined with Example \ref{rP1} that for all positive integers $n$ the $\C\Pb^1$-bundles
$X \times \C\Pb^1$ and $Y_{\mathcal{L}^n}$ are bimeromorphic, and therefore the groups $\Bim(X\times \C\Pb^1)$ and
$\Bim(Y_{\mathcal{L}^n})$  are isomorphic.
\item[(iii)]
It follows from (i) and (ii) that for all positive integers $n$ the group $\Bim\left(Y_{\mathcal{N}}\right)$ contains a subgroup 
isomorphic to $\mathbf{S}(\mathcal{N}\otimes \mathcal{L}^n)$ and the group $\Bim(X\times \C\Pb^1)$ contains a subgroup
isomorphic to  $\mathbf{S}(\mathcal{L}^n)$. 
 We will use this observation  (together with Theorem \ref{pi0}) in the proof of Theorems \ref{manP1} and \ref{manP2}.
\end{itemize}
\end{rems}

\label{firstproof}
\begin{proof}[Proof of Theorem \ref{manP1}]
Since $X$ has {\sl positive} algebraic dimension, it follows from the results of \cite[Ch. 2, Sect. 6]{BL} that there exists a surjective holomorphic homomorphism $\psi:X \to A$ to a positive-dimensional complex abelian variety $A$. There exists a very ample holomorphic line bundle $\mathcal{M}$ on $A$ such that the group $\mathrm{H}^0(A,\mathcal{M})$ of global sections of $\mathcal{M}$
has $\C$-dimension at least $2$. Let us consider the induced   holomorphic line bundle $\psi^{*}\mathcal{M}$ on $X$.
Since $\psi$ is surjective,  the group $\mathrm{H}^0(X,\psi^{*}\mathcal{M})$ of global sections of $\psi^{*}\mathcal{M}$ also has $\C$-dimension at least $2$,
because $\mathrm{H}^0(A,\mathcal{M})$ embeds into  $\mathrm{H}^0(X,\phi^{*}\mathcal{M})$.
 There exists an A.-H. data $(H,\alpha)$    on $X$ such $\psi^{*}\mathcal{M}$ is isomorphic to $\LL(H,\alpha)$.
This implies that $\LL(H,\alpha)$ has at least two linearly independent nonzero holomorphic sections. Now if $H=0$ then
$\LL(H,\alpha)=\LL(0,\alpha)$ and one of the following two conditions holds:

\begin{itemize}
\item[(i)]
 $\alpha=\alpha_0$, i.e., $\LL(H,\alpha)=\LL(0,\alpha_0)$ is isomorphic to $\mathbf{1}_X$ 
and $$\mathrm{H}^0(X,\LL(H,\alpha))=\mathrm{H}^0(X,\LL(0,\alpha_0))=\mathrm{H}^0(X, \mathbf{1}_X)=\C.$$
\item[(ii)]
$\alpha \ne \alpha_0$. It follows from \cite[Th. 2.1]{Kempf} that 
$$\mathrm{H}^0(X,\LL(H,\alpha))=\mathrm{H}^0(X,\LL(0,\alpha))=\{0\}.$$
\end{itemize}
Since the $\C$-dimension of $\mathrm{H}^0(X,\LL(H,\alpha))$ is at least $2$,
neither (i) nor (ii) holds.
This implies  that   $H\ne 0$.

 Let $n$ be a positive integer. Then $nH\ne 0$ and the holomorphic line bundle
$$\mathcal{L}(nH,\alpha^n)\cong\mathcal{L}(H,\alpha)^{\otimes n}$$
over $X$
also admits a nonzero holomorphic section.  Notice that
$$E_n:=\IM(nH)=nE \ \text{ where } E=\IM(H).$$
In particular,
$E_n(L,L)\subset n\Z$.
It follows from  Corollary \ref{JordanN} applied to $\mathcal{L}(nH,\alpha^n)$  that the Jordan constant of $\G\left(nH,\alpha^n\right)$ is at least $n$. By Theorem \ref{embedP1}, there is a
 group embedding 
 $$\G(nH, \alpha^n)\hookrightarrow \Hol\left(\Pb\left(\mathcal{L}(nH,\alpha^n\right)\oplus  \mathbf{1}_X\right)).$$ 
 By Remark \ref{section}, $\Bim(X\times \C\Pb^1)$ and $\Bim\left(\Pb\left( \mathcal{L}(nH,\alpha^n)\oplus \mathbf{1}_X\right)\right)$ are isomorphic.  This implies  that for all positive integers $n$ the group $\Bim(X\times \C\Pb^1)$ contains a subgroup, whose Jordan constant 
 is at least $n$.
 It follows that $\Bim(X\times \C\Pb^1)$ is {\sl not} Jordan.
\end{proof}

\section{Pencils of Hermitians forms}
\label{pencils}
In order to prove Theorem \ref{manP2}, we need to construct families of Hermitian forms and corresponding  alternating forms. We keep the notation and assumptions of Section \ref{AppelHumbert}.

Let $\HH: V \times V\to\C$ be an Hermitian form. Let us consider its imaginary part
$$\EE: V \times V \to \R, \ (u,v) \mapsto \IM(\HH(u,v)),$$
which is an alternating $\R$-bilinear form on $V$.
Let us assume that
$$\EE(L,L) \subset \Z.$$

\begin{defn}
We say that $\HH$ is dominated by $H$ if
$$\ker(H)\subset \ker(\HH).$$
\end{defn}

\begin{sect}
\label{dom}
For every positive integer $n$ let us consider the Hermitian form
$$H_n:=H+n\HH: V \times V \to \C,$$
 whose  imaginary part
 $$E_n:=\IM(H_n)=E+n\EE: V \times V \to \R$$
 is an alternating $\R$-bilinear form on $V$.
 Clearly, for all $n$
 $$E_n(L,L) \subset \Z.$$
 If $\HH$ is dominated by $H$ then every $H_n$ is also dominated by $H$.
 \end{sect}
 
 \begin{thm}
 \label{pencilH} 
 Suppose that $H\ne 0$  (i.e., $g>g_0$) and that $\HH$ is dominated by $H$.
 Then
 for all but finitely many $n$ 
 \begin{equation}
 \label{kerHn}
 \ker(H)=\ker(\HH)
 \end{equation}
 and the restriction
 \begin{equation}
 \label{EnND}
 E_n\bigm|_{L_1}: L_1\times L_1 \to \Z
 \end{equation}
 of $E_n$ to $L_1$ is a nondegenerate alternating bilinear form.
 \end{thm}
 
 \begin{proof}
 Let $\tilde{\EE}$ be the  square matrix of $\EE\bigm|_{L_1}$ of order  $2g-2g_0$
 with respect to the basis $\{\bar{l}_1, \dots, \bar{l}_{2g-2g_0}\}$ of $L_1$.  (Recall that  $\tilde{E}$ is the matrix of $E\bigm|_{L_1}$ with respect to the same basis.)
 Then
 for all $n$ the matrix $\tilde{\EE}+n\tilde{E}$ coincides with  the  matrix $\tilde{E}_n$ of $E_n\bigm|_{L_1}$.
 with respect to $\{\bar{l}_1, \dots, \bar{l}_{2g-2g_0}\}$. Recall that the matrix $\tilde{E}$ is nondegenerate and
  consider the polynomial
 \begin{equation}
 \label{detPo}
 f_{H,\HH}(T): =\det(\tilde{E})\cdot \det\left(\tilde{E}^{-1}\tilde{\EE}+T\right)\in \Q[T].
 \end{equation} 
 Clearly, $f_{H,\HH}(T)$ is a degree $2g-2g_0$ polynomial with (positive) leading coefficient 
 $\det(\tilde{E})$.
  We have
 $$\det(\tilde{E}_n)=\det(\tilde{\EE}+n\tilde{E})=\det\left(\tilde{E}\left(\tilde{E}^{-1}\tilde{\EE}+n\right)\right)=$$
 $$\det(\tilde{E})\det\left(\tilde{E}^{-1}\tilde{\EE}+n\right)= f_{H,\HH}(n).$$
 In other words,
 \begin{equation}
 \label{detPol}
  \det(\tilde{E}_n)= f_{H,\HH}(n).
  \end{equation}
 Since $f_{H,\HH}(T)$ is {\sl not} a constant, $\det(\tilde{E}_n)\ne 0$ for all but finitely many $n$.
  
 Let us assume that $\det(\tilde{E}_n)\ne 0$, which is true for all but finitely many positive integers $n$. Then
  $E_n\bigm|_{L_1}$ is {\sl nondegenerate}. It follows that  the restriction of $E_n$ to $L_1\otimes\R$ is nondegenerate as well.
 On the other hand, the restriction of  $E_n$  to  $\ker(H)$ is identically $0$. This implies  that $\ker(E_n)=\ker(H)$
 and therefore $$\ker(H_n)=\ker(E_n)=\ker(H).$$ 
 \end{proof}
 
 \begin{defn}
 \label{perpDef}
 Suppose that $H \ne 0$ and $n$ is a positive integer such that $\ker(H)=\ker(\HH)$  and   $E_n\bigm|_{L_1}$ is a {\sl nondegenerate}  (By Theorem \ref{pencilH}, these properties hold for all but finitely many positive integers $n$.)  Let us define
 $L_{1,E_n}^{\perp}$ as
 $$L_{1,E_n}^{\perp}=\{\tilde{x}\in  L_1\otimes\R\mid E_n(\tilde{x},l)\in \Z \ \forall l\in L\}.$$
 \end{defn}
 
 \begin{rem}
 \label{quotientPerp}
 
 Applying arguments of  Subsection \ref{normalForm1} to
  {\sl nondegenerate} $E_n\bigm|_{L_1}$  (instead of $E\bigm|_{L_1}$), we obtain that
 $L_{1,E_n}^{\perp}$ is a free $\Z$-module of rank $2g-2g_0$ that lies in $L_1\otimes \Q$ and contains $L_1$ as a subgroup of finite index, i.e.,
 the quotient  $L_{1,E_n}^{\perp}/L_1$ is a finite commutative group.
 
 It follows from  \eqref{detPol}  and the arguments of Subsection \ref{normalForm1} applied to $E_n$ (instead of $E$) that
 \begin{equation}
 \#(L_{1,E_n}^{\perp}/L_1)=\det(\tilde{E}_n)= f_{H,\HH}(n).
 \end{equation}
 Since  $ f_{H,\HH}(T)$ is a polynomial of positive degree,  if $n$ tends to infinity
 then $\#(L_{1,E_n}^{\perp}/L_1)$ also tends to infinity, i.e., 
 %.  Hence if $n$ tends to infinity then 
 $\sqrt{\#(L_{1,E_n}^{\perp}/L_1)}$  tends to infinity as well.
 \end{rem}
 
 \begin{thm}
 \label{JordanP}
 Let $H \ne 0$ be a semi-positive Hermitian form, $\HH$ a Hermitian form that is dominated  by $H$,
  $(\HH,\alpha)$  an A.-H. data, $\mathcal{L}(\HH,\alpha)$ the corresponding holomorphic line bundle on $X$
  and $Y_{\mathcal{L}(\HH,\alpha)}$ the corresponding $\C\Pb^1$-bundle on $X$.
  
  Then the group $\Bim(Y_{\mathcal{L}(\HH,\alpha)})$ is not Jordan.
 \end{thm}
 
 \begin{proof}
 Replacing $H$ by $2H$, we may and will assume that its imaginary part $E$ satisfies
 $E(L,L)\subset 2\Z$.
% If $$\alpha_0: L \to \{1\}\subset \UU(1)$$
% is a trivial group homomorphism 
 Then $(H,\alpha_0)$ is an A.-H. data. Since $H$ is semi-positive,
 it follows from \cite[Th. 2.1 on p. 9]{Kempf}
 that the holomorphic line bundle
 $\mathcal{L}(\HH,\alpha_0)$ admits a {\sl nonzero} holomorphic section.  Since for all positive integers $n$
 $$H_n=\HH+n H,  \alpha=\alpha\cdot \alpha_0^n,$$
 we obtain that
 $$\mathcal{L}(H_n,\alpha)\cong\mathcal{L}(\HH,\alpha)\otimes \mathcal{L}(H,\alpha_0)^n.$$
 It follows from Remark \ref{section} that the groups $\Bim(Y_{\mathcal{L}(\HH,\alpha)})$ and
 $\Bim(Y_{\mathcal{L}(H_n,\alpha)})$ are isomorphic.
 On the other hand,  by Theorem \ref{embedP1}, $\Aut(Y_{\mathcal{L}(H_n,\alpha)})$ contains a subgroup isomorphic to $\G(H_n,\alpha)$.
 In light of Theorems \ref{pencilH} and   Theorem \ref{JordanT} (applied to $(H_n,\alpha)$),  for all sufficiently large $n$
 the Jordan constant of $\G(H_n,\alpha))$ is  $\sqrt{\#(L_{1,E_n}^{\perp}/L_1)}$. It follows from Remark \ref{quotientPerp}
 that the Jordan constant of $\G(H_n,\alpha)$ tends to infinity when $n$ tends to infinity.  Since each  
 $$\G(H_n,\alpha) \hookrightarrow \Hol(Y_{\mathcal{L}(H_n,\alpha)}) \subset \Bim(Y_{\mathcal{L}(H_n,\alpha)})$$
 is isomorphic to a certain subgroup of $\Bim(Y_{\mathcal{L}(\HH,\alpha)})$, we conclude that the Jordan constant of 
 $\Bim(Y_{\mathcal{L}(\HH,\alpha)})$ is $\infty$, i.e., $\Bim(Y_{\mathcal{L}(\HH,\alpha)})$ is {\sl not} Jordan.
 \end{proof}

\section{Complex tori and abelian varieties}
\label{secondproof}
\begin{sect}
\label{AV}
A complex abelian variety $A$ of positive dimension  is a complex torus $W/\Gamma$ where $W$ is a $\C$-vector space of finite positive dimension
and $\Gamma\subset W$ is a discrete additive group of maximal rank $2\dim_{\C}(W)$. In addition, there exists a {\sl polarization},
i.e., a {\sl positive-definite} Hermitian form
$$H_A:W \times W \to \C$$
such that 
$$\IM(H_A(\gamma_1,\gamma_2))\in \Z \ \forall \gamma_1,\gamma_2 \in \Gamma.$$
\end{sect}
\begin{proof}[Proof of Theorem \ref{manP2}]
Every surjective holomorphic homomorphism $\psi:X \to A$ is induced by a certain surjective $\C$-linear map
$\bar{\psi}: V \to W$  such that  $\psi(L)\subset \Gamma$
in the sense that
$$\psi(v+L)=\bar{\psi}(v)+\Gamma\in W/\Gamma=A \ \forall v+L\in V/L=X.$$
Every holomorphic line bundle $\mathcal{M}$ on $A$ is isomorphic to $\mathcal{L}(\HH_A,\beta)$
for a certain  A.-H. data  $(\HH_A,\beta)$ where the Hermitian form
$$\HH_A: W \times W\to \C$$ satisfies
$$\EE_A(\gamma_1,\gamma_2):=\IM(\HH_A(\gamma_1,\gamma_2))\in \Z \ \forall \gamma_1,\gamma_2 \in \Gamma$$
and the map $\beta:\Gamma \to \UU(1)$ satisfies
$$\beta(\gamma_1+\gamma_2)=(-1)^{\EE_A(\gamma_1,\gamma_2)}\beta(\gamma_2)\beta(\gamma_1) \ \forall \gamma_1,\gamma_2 \in \Gamma.$$
In addition, it follows from \cite[Lemma 2.3.4 on p. 33]{BLA} that the induced holomorphic line bundle $\psi^{*}\mathcal{M}$ on $X$ is isomorphic to 
$\mathcal{L}(\HH,\alpha_1)$ where
$$\HH:V \times V \to \C,  \ \HH(v_1,v_2)= \HH_A(\bar{\psi}(v_1),\bar{\psi}(v_2));$$
$$\alpha_1: L \to \UU(1), \ \alpha_1(l)=\beta(\bar{\psi}(l)).$$
Clearly,
$$\ker(\bar{\psi})\subset \ker(\HH)\subset V.$$
Let us choose a  polarization $H_A$ on $A$
 and consider the  Hermitian form
$$H: V \times V \to \C,  \ H(v_1,v_2)= H_A(\bar{\psi}(v_1),\bar{\psi}(v_2)).$$
Clearly, $H\ne 0$, it is semi-positive and for all $l_1,l_2\in L$
$$\IM(H(l_1,l_2))=\IM(H_A(\bar{\psi}(l_1), \bar{\psi}(l_2))\in \Z,$$
because $\bar{\psi}(l_1), \bar{\psi}(l_2)\in \Gamma$. On the other hand, since $H_A$ is positive and therefore nondegenerate,
$\ker(H)=\ker(\bar{\psi})$. This implies that 
$\ker(H)\subset  \ker(\HH)$,
 i.e., $\HH$ is {\sl dominated} by $H$.
It follows from Theorem \ref{JordanP} that the group $\Bim(Y_{\mathcal{L}(\HH,\alpha)})$ is {\sl not} Jordan
for every holomorphic line bundle $\mathcal{L}(\HH,\alpha)$ where $\alpha: L \to \UU(1)$ is any map such that
$(\HH,\alpha)$ is an A.-H. data.  On the other hand, every  holomorphic line bundle on $X$  that is isomorphic to
$\mathcal{L}(\HH,\alpha_1)\otimes \mathcal{F}_0$ with  $\mathcal{F}_0\in \Pic^0(X)$ is isomorphic to 
$\mathcal{L}(\HH,\alpha)$ for suitable $\alpha$. In order to finish the proof, one has only to recall that 
$\mathcal{L}(\HH,\alpha_1)$ is isomorphic to  $\psi^{*}\mathcal{M}$.
\end{proof}

\begin{proof}[Proof of Theorem \ref{manP3}]
A nonzero complex subtorus $X_0\subset X$  and the quotient $A=X/X_0$ admit the following description.
There exists a nonzero $\C$-vector subspace $U\subset V$ such that $L_U=L\bigcap U$ is a  lattice
of rank $2\dim_{\C}(W)$ in $U$, the quotient $L/L_U$ is a  lattice of rank $2\dim_{\C}(V/U)$ in the nonzero $\C$-vector space $W:=V/U$ and
$$X_0=U/L_U\subset V/L=X, \ A=(V/U)/(L/L_U)=W/\Gamma$$
where 
$$ \Gamma:=L/L_U\subset V/U=W.$$
We may assume that $\mathcal{F}=\mathcal{L}(\HH,\alpha)$ for a certain A.-H. date $(\HH,\alpha)$ on $X$ where $\HH$ is an Hermitian form
$$\HH: V \times V \to \C,$$
whose imaginary part
$$\EE:=\IM(\HH): V \times V \to \R$$
is integer valued on $L\times L$. The   restriction of $\mathcal{F}$ to $X_0$ lies in $\Pic^0(X_0)$. It follows from \cite[Lemma 2.3.4 on p. 33]{BLA}
 that
$$\HH(U,U)=\{0\}.$$
This implies that  $\HH$ induces the biadditive form
$$S: U\times W(=V/U) \to \C, \ S(u,v+U)=\HH(u,v)$$
such that
$$S(\lambda u,w)=\lambda S(u,w), \  S(u,\lambda w)=\bar{\lambda} \cdot S(u,w)$$
for all $u\in U, w\in W, \lambda\in \C$. In addition,
\begin{equation}
\label{ImS}
\IM(S(l,\gamma))\in \Z \ \forall l\in L_U\subset U, \gamma \in \Gamma\subset W.
\end{equation}
Clearly, 
\begin{equation}
\label{HkerH}
S=0 \ \text{ if and only if } U\subset \ker(\HH).
\end{equation}
Let us consider the $\dim_{\C}(W)$-dimensional $\C$-vector space $\Hom_{\text{anti-lin}}(W,\C)$ of  $\C$-antilinear
maps
$$h: W \to \C, \  h(w_1+w_2)=h(w_1)+h(w_2), h(\lambda w)=\bar{\lambda}\cdot h(w) \ \forall w_1,w_2,w\in W, \lambda\in \C$$
and the  lattice
$$\Gamma_{\text{anti-lin}}:=\{h\in \Hom_{\text{anti-lin}}(W,\C)\mid \IM(h(\gamma))\in \Z \ \forall \gamma \in \Z\}\subset \Hom_{\text{anti-lin}}(W,\C)$$
of  rank $2\dim_{\C}(W)$.
The form $S$ defines the $\C$-linear homomorphism of vector spaces
$$a_S: U \to \Hom_{\text{anti-lin}}(W,\C),  \ u \mapsto \{w \mapsto S(u,w)\}.$$
Clearly, 
\begin{equation}
\label{SaS}
S=0 \ \text{ if and only if } a_S=0.
\end{equation}
In light of \eqref{ImS}, 
$a_U(L_U)\subset \Gamma_{\text{anti-lin}}$.
This implies that $a_S$ induces a holomorphic homomorphism of complex tori
$$b_S:U/L_U \to \Hom_{\text{anti-lin}}(W,\C)/\Gamma_{\text{anti-lin}}, \ u+L_U \mapsto a_S(u)+\Gamma_{\text{anti-lin}}.$$
Recall that $U/L_U=X_0$ and   $W/\Gamma$ is our complex abelian variety $A$. It is proven in  \cite[Sect. 3]{Mumford} that
$\Hom_{\text{anti-lin}}(W,\C)/\Gamma_{\text{anti-lin}}$ is the {\sl dual abelian variety} $\hat{A}$ of $A$.  We are given that
$\Hom(X_0,A)=\{0\}$. Since every abelian variety and its dual are isogenous, $\Hom(X_0,\hat{A})=\{0\}$ as well. It follows that
$b_S=0$. This means that  the $\C$-vector subspace  $a_S(U)$ lies in the  lattice $\Gamma_{\text{anti-lin}}$ and therefore $a_S=0$. By \eqref{SaS}, $S=0$.
Now it follows from \eqref{HkerH} that 
$U \subset \ker(\HH)$.
This implies that there is an Hermitian form 
$\HH_A: W \times W \to \C$
on $W =V/U$ such that
\begin{equation}
\label{HHA}
\HH_A(v_1+U, v_2+U)=\HH(v_1,v_2) \ \forall \ v_1, v_2\in V; \ v_1+U, v_2+U \in V/U=W.
\end{equation}
We have
$$\IM\left(\HH_A(l_1+L_U, l_2+L_U)\right)=\IM(\HH(l_1, l_2))\in \Z \ \forall
l_1,l_2\in L; l_1+L_U, l_2+L_U \in  L/L_U=\Gamma.$$
By \cite[Sect. 1.4, Lemma 1.6]{Kempf}, there exists a map $\beta:\Gamma \to \UU(1)$ such  that
$(\HH_A,\beta)$ is an A.-H. data on $A$. Let $\LL(\HH_A,\beta)$ be the corresponding holomorphic line bundle
on $A$.  The {\sl inverse image} $\psi^{*}\LL(\HH_A,\beta)$ on $X$ is a holomorphic line bundle on $X$ that
is isomorphic to some $\LL(\HH^{\prime},\alpha^{\prime})$.  It follows from \cite[Lemma 2.3.4 on p. 33]{BLA} that
 the Hermitian form $\HH^{\prime}$ on $V$ and the map $\alpha^{\prime}:L\to \UU(1)$ are as follows:
\begin{equation}
\label{Hprime}
\HH^{\prime}(v_1,v_2)=\HH_A(v_1+U,v_2+U) \ \forall v_1,v_2\in U.
\end{equation}
$$\alpha^{\prime}(l)=\beta(l+L_U) \ \forall l\in L.$$
It follows from \eqref{Hprime} and \eqref{HHA} that $\HH^{\prime}=\HH$. This means that $\psi^{*}\LL(\HH_A,\beta)$
is isomorphic to $\LL(\HH,\alpha_1)$. Since $\mathcal{F}=\mathcal{L}(H,\alpha)$, it is isomorphic
to $\psi^{*}\LL(\HH_A,\beta)\otimes \mathcal{F}_0$ where 
$\mathcal{F}_0=\LL(0, \alpha/\alpha_1) \in \Pic^0(X)$.
Now the desired result follows from Theorem \ref{manP3}.
\end{proof}

\section{$\Pic^0$ and theta groups}
\label{Pic0}
In this section we revisit theta groups that correspond to 
the case $H=0$. 
The main idea is to identify the theta group of a line bundle from $\Pic^{0}$ and the total space of the bundle with zero section removed.
(See \cite{Zarhin,ZarhinTG}  where the case  of abelian varieties was discussed.)

Recall  that a holomorphic line bundle $\LL$ over $X$ lies in  $\Pic^0(X)$ if the corresponding Hermitian form $H$ is zero, i.e., $\LL\cong L(0,\alpha)$. 
If this is the case then $$\alpha:L \to \UU(1)\subset \C^{*}$$ is a {\sl group homomorphism} and $\LL(0,\alpha)$ is the quotient
of the direct product $V \times \C$ modulo the following action of $L$.
$$(v,c) \mapsto (v+l, \alpha(l) c) \ \forall l \in L; \ v \in V, c\in \C.$$
On the other hand, the $\C^{*}$-bundle $\LL(0,\alpha)^{*}$ over $X$ obtained from $\LL(0,\alpha)$
by removing zero section may be viewed as the quotient $\left(V\times \C^{*}\right)/\tilde{L}$ of the {\bf commutative complex  Lie group}
$V\times \C^{*}$ by its {\bf discrete subgroup} 
$$\tilde{L}:=\{(l,\alpha(l))\mid l\in L\}\subset L \times \C^{*}.$$
In particular, $\LL(0,\alpha)^{*}$ carries the natural structure of a commutative complex Lie group. It is included in the short exact sequence of commutative complex Lie groups
$$1 \to \C^{*} \to \LL(0,\alpha)^{*} \to (V/L=) X \to 0.$$
Notice that the natural  faithful action of $V\times \C^{*}$ on $V \times \C$ descends to the faithful action
of $\LL(0,\alpha)^{*}$ on $\LL(0,\alpha)$, so one may view $\LL(0,\alpha)^{*}$ as a subgroup of 
%the  group 
$\Hol(\LL(0,\alpha))$. 
%of holomorphic automorphisms of the total space of $\LL(\alpha,0)$.
\begin{comment}
Recall that the subgroup $\mathbf{S}(\LL(0, \alpha))\subset \Hol(\LL(0,\alpha))$  is the set of all  $\mathfrak{u}\in  \Hol(\LL(0,\alpha))$
that enjoy the following properties.
\begin{itemize}
\item[(i)]
There exists $x\in X$ such that $\mathfrak{u}: \LL(0,\alpha) \to \LL(0,\alpha)$ is a lifting of $T_x:X \to X$.
\item[(ii)]
For each $z \in X$ the map between the fibers of  $\LL(0,\alpha)$ over $z$ and $z+x$ induced by $\mathfrak{u}$ is a linear
isomorphism of one-dimensional $\C$-vector spaces.
\end{itemize}
\end{comment}
\begin{rem}
\label{scalar}
%\begin{itemize}
%\item[(i)]
Clearly,  
$\LL(0,\alpha)^{*}\subset \mathbf{S}(\LL(0,\alpha))\subset \Hol(\LL(0,\alpha))$ and
%\item[(ii)]
for each $c\in \C^{*}$ 
$$(0,c)\tilde{L}\in \LL(0,\alpha)^{*}\subset \mathbf{S}(\LL(0, \alpha))\subset \Hol(\LL(0,\alpha))$$
acts as multiplication by $c$ in all fibers of $\LL(0,\alpha)\to X$.
%\end{itemize}
\end{rem}

\begin{thm}
\label{main}
$\LL(\alpha,0)^{*}= \mathbf{S}(\LL(0, \alpha))$. In particular, $\mathbf{S}(\LL(0,\alpha))$ is commutative.
\end{thm}

\begin{proof}
Let $\mathfrak{u} \in \mathbf{S}(\LL(0, \alpha))$.  Then there is $y \in X$ such that $\mathfrak{u}$ is a lifting of $T_y$.
Choose $\tilde{y} \in \LL(0,\alpha)^{*}$ that lifts $T_y$ as well. For example, take $v \in V$
such that $y=V+L$ and put
$$\tilde{y}=(v,1)\tilde{L}\in \left(V\times \C^{*}\right)/\tilde{L}.$$
Then $\mathfrak{u}\tilde{y}^{-1}$ is an automorphism of $\LL(0,\alpha)$ that sends every fiber of $\LL(0,\alpha)\to X$ into itself
and acts on each  such fiber as a $\C$-linear automorphism. This means that there is a holomorphic function 
$f$ on $X$  that does not vanish and such that $\mathfrak{u}\tilde{y}^{-1}$ acts on the fiber $\LL(0, \alpha)_z$
as the multiplication by $f(z)\in \C^{*}$ for all  $z\in X$. The compactness and connectedness of $X$ implies that there is $c \in \C^{*}$
such that $f(z)=c$ for all $z\in X$. It follows from Remark \ref{scalar}(ii) that
$u\tilde{y}^{-1}\in \LL(0,\alpha)^{*}$. Since $\tilde{y}\in \LL(0,\alpha)^{*}$, we have
$\mathfrak{u}=\left(\mathfrak{u}\tilde{y}^{-1}\right)\tilde{y}\in  \LL(0,\alpha)^{*}$.
This completes the proof.
\end{proof}

%\begin{rem}
%In the case of abelian varieties the assertion of Theorem \ref{main} is  well known and essentially  becomes a classical theorem of M. Rosenlicht
%\cite{Serre,Mumford}. 
%\end{rem}

\end{document}